\documentclass{article}

\usepackage{amsmath,amssymb,,amsthm}
\usepackage{graphicx}
\usepackage{comment}

\theoremstyle{plain}
\newtheorem{theorem}{Theorem}
\newtheorem*{theorem*}{Theorem}
\newtheorem{lemma}{Lemma}

\newtheorem*{corollary*}{Corollary}
\newtheorem{proposition}{Proposition}
\newtheorem{definition}{Definition}
\newtheorem*{remark*}{Remark}
\newtheorem*{conjecture*}{Conjecture}
\newtheorem{conjecture}{Conjecture}
\newtheorem{question}{Question}

\DeclareMathOperator\Ln{Ln}
\DeclareMathOperator\Log{Log}

\newcommand{\ov}[1]{\overline{#1}}
\def\ds{\displaystyle}

\def\R{{\mathbb R}}

\def\C{{\mathbb C}}
\def\Q{{\mathbb Q}}
\def\oQ{\overline{\mathbb Q}}
\def\SA{S\!A}
\def\EL{E\!L}

\begin{document}

\title{Can the quadratrix truly square the circle?}

\author{Luis Cruz$^*$ and Sergiy Koshkin$^\dagger$\\
\\
*YES Prep Public Schools\\
Northbrook High\\
1 Raider Circle\\
Houston, TX 77080\\
e-mail: \texttt{luasrey@gmail.com}\\
\\
\textdagger Corresponding author\\
University of Houston-Downtown\\
Department of Mathematics and Statistics\\
1 Main Street\\
Houston, TX 77002\\
e-mail: \texttt{koshkins@uhd.edu}
}

\date{}

\maketitle

\begin{abstract}
The quadratrix received its name from the circle quadrature, squaring the circle, but it only solves it if completed by taking a limit, as pointed out already in antiquity. We ask if it can square the circle without limits and restrict its use accordingly, to converting ratios of angles and segments into each other. The problem is then translated into algebra by analogy to straightedge and compass constructions, and leads to an open question in transcendental number theory. In particular, Lindemann's impossibility result no longer suffices, and the answer depends on whether $\pi$ belongs to the analog of Ritt's exponential-logarithmic field with an algebraic base. We then derive that it does not from the well-known Schanuel conjecture. Thus, the quadratrix so restricted cannot square the circle after all.


\end{abstract}

\section{Geometric introduction}

The problem of squaring the circle, {\it quadrature} for short, has become a paradigm of the unsolvable problem since it was posed by the ancient Greeks over 2,500 years ago. They already guessed that solving the quadrature with straightedge and compass was impossible and looked for additional tools to do it. One of them was a curve renamed after the problem when Dinostratus used it to this effect---the quadratrix. It was generated mechanically, by composing uniform linear and circular motions, and used to construct a segment equal to the circle's circumference. This is now called {\it rectification} (straightening out) of the circle. The circle could then be squared by a straightedge and compass construction known, at least, since Archimedes. Or so it seemed.

Here is what is less known. Philosopher Sporus objected to the solution with the quadratrix already in antiquity. And ``with good reason", remarked Pappus of Alexandria in his {\it Collection}, one of our main sources on ancient mathematics \cite{Sefrin}. To make a long story short (see Section \ref{Sporus}), the mechanically generated part of the quadratrix {\it does not} produce a segment equal to the circle's circumference. One has to take a limit along the curve to get it. This is cheating. If the task is an exact construction of the square with the circle's area, the ancient Greek construction with the quadratrix does not accomplish it. Moreover, as we will explain, if limits are allowed one can produce approximating segments with straightedge and compass alone: the quadratrix is not needed at all\,!

The question then becomes: is it possible to use the quadratrix, together with straightedge and compass, to square the circle {\it without} chea... taking limits? Classical uses of the quadratrix that meet this condition include dividing the right angle in a given ratio and, conversely, dividing a segment in the same ratio that a given acute angle divides the right angle. Let us call the corresponding tools the {\it right anglesector} and the {\it reverse right anglesector}, respectively. Combining them with straightedge and compass, one can, in particular, perform the {\it general anglesection}, i.e., divide any given angle in the ratio of any two given segments. Unlike the somewhat vague idea of ``using the quadratrix," using these tools leads to a precise mathematical question.
\begin{question}\label{QRRA} Can one square the circle with straightedge, compass, right anglesector, and reverse right anglesector?
\end{question}
Surprisingly (or not), this question is still open today, but the answer is likely negative. The reason why leads us to the fascinating and actively studied field of transcendental number theory, exponential algebra, and its central open conjecture, the Schanuel conjecture \cite{Kir10,Terzo,Zil05}. As far as we know, relations between exponential algebra and geometry have not been explored before.

\section{Algebraic introduction}

Before the impossibility of the quadrature was proved, the problem had to be translated from geometry to algebra, mostly by Ren\'e Descartes and Carl Friedrich Gauss \cite{Hart,LutzenW,Lutzen}. The correct translation is subtle as it has to take into account the iterative nature of straightedge and compass constructions, a point missed by Descartes. We refer to \cite{Hart,JMP,Waerden} for detailed discussions. 

According to the translation (see Section \ref{Numb}), geometric constructions are paralleled by operations on real numbers: field operations and taking square roots of positive numbers. The numbers that can be produced by iterating those, starting from the rational numbers, are called {\it constructible}. Some examples of constructible numbers are $\sqrt{2}$, $\frac{1+\sqrt{5}}{2}$ (golden ratio), and $\sqrt{5-\frac18\sqrt{3}}$. The question about squaring the circle with straightedge and compass becomes the question of whether $\pi$ is constructible. It is to this question that Ferdinand von Lindemann gave a negative answer in 1882. 

His answer relied on subtle algebraic ideas and techniques developed by Niels Henrik Abel, \'Evariste Galois, Joseph Liouville and Charles Hermite, among others \cite{Niven}. Constructible numbers belong to a much larger class of {\it algebraic} numbers, roots of polynomial equations with rational coefficients. Algebraic numbers include even some complex numbers, like the imaginary unit $i$, which is a root of $x^2+1$. The remaining complex numbers are called {\it transcendental}, and Lindemann proved that $\pi$ is transcendental. This is much stronger than was needed, and covers even constructions with additional classical tools, like the marked straightedge \cite{Baragar}, or the angle trisector \cite{Gleason}.

We will show in Section \ref{Numb} what new operations the new tools add to the field operations and square roots. The right anglesector generates $\sin(\pi x)$ for real $x$, and the reverse right anglesector generates $\frac1\pi\arcsin(x)$ for real $x$ with $|x|\leq1$. The question becomes if adding those is enough to generate $\pi$ starting from the rational numbers. It is ironic that the question about generating $\pi$ involves functions with $\pi$ in their defining formulas, but that is only because we use radians for angles when defining $\sin$ and $\arcsin$. If we used $\sin$ and $\arcsin$ for angles measured in degrees, as in \cite{Niven}, then $\pi$ would not appear in the formulas. Still, one can easily see that the standard ways of obtaining $\pi$ as, say, $6\arcsin\left(\frac12\right)$, are not available. At the same time, already numbers like $\sin(\pi\sqrt{2})$ or $\frac1\pi\arcsin\left(\frac13\right)$ are transcendental, so Lindemann's result is insufficient to answer Question \ref{QRRA}. 

Trigonometric functions and their inverses are typically not included in algebra, but in the complex domain they can be reduced to exponential and logarithmic functions. This is a consequence of the Euler's famous formula $e^{ix}=\cos x+i\sin x$. As a result, we can handle them algebraically if we add exponentiation to the field operations on $\C$, a homomorphism from its additive to its multiplicative group \cite{Kir10,Zil05}. The standard choice is $e^x$, but $b^x$ with other $b$ is also an option. It can be defined as $e^{x\ln b}$ by fixing a value of $\ln b$, e.g., the principal value of the logarithm. We are interested in $b=-1$ because $\ln\,(-1)=i\pi$, and so
\begin{equation}\label{exp-1}
(-1)^x:=e^{i\pi x}=\cos(\pi x)+i\sin(\pi x).
\end{equation}
Therefore, $\sin(\pi x)=\frac{(-1)^x-\,(-1)^{-x}}{2i}$, and its inverse, $\frac1\pi\arcsin(x)$, can also be expressed algebraically in terms of $\log_{\,-1}(x):=\frac1{i\pi}\ln x$. 

In the spirit of going from constructible to algebraic numbers, we then build a field by starting from the rational numbers, adjoining roots of polynomials, $(-1)^x$, $\log_{\,-1}(x)$ for every $x$, and iterating those operations (Section \ref{ExpLog}). It turns out that using any other algebraic base $b\neq0,1$ produces the same field as $b=-1$. We denote it $\EL^{alg}$ and call its elements {\it algebraically based} numbers. More conceptually, they can be described as the numbers obtained from the rational numbers by repeatedly taking algebraic closures, and exponents and logarithms of previously generated numbers. The negative answer to Question \ref{QRRA} is then implied by the following conjecture that strengthens Lindemann's result.
\begin{conjecture}\label{ConEL} $\pi$ is not algebraically based.
\end{conjecture}
An analogously constructed field with $b=e$ was introduced by Joseph Ritt, who called its elements ``elementary numbers" by analogy to Liouville's elementary functions \cite{Chow,Lin,Ritt}. Clearly, $\pi=\frac1i\ln\,(-1)$ is ``elementary", but questions about concrete numbers {\it not} being ``elementary" are still open. Nonetheless, many of them can be answered by making one big assumption, the Schanuel conjecture (Section \ref{Schan}). Albeit less known, it plays a role in transcendental number theory similar to the role of the Riemann hypothesis in ordinary number theory---it is believed to be true by most experts and results are commonly proved conditionally on it \cite{Cheng,Chow,Lang71,Lin,Terzo}. 

In the following sections, we will describe the classical rectification and quadrature with the quadratrix, Sporus's objections to it, anglesector tools and their algebraic counterparts, some elementary exponential algebra, and, finally, derive Conjecture \ref{ConEL} from the Schanuel conjecture. In fact, we will derive a stronger claim (Theorem \ref{pinEL}), which may be of independent interest, that many other Ritt's ``elementary" numbers are also not algebraically based. Applied to $e$, it strengthens (conditionally) Hermite's result that $e$ is not algebraic. Even if the Schanuel conjecture is false, the answer to Question \ref{QRRA} may still be negative, but it would be really exciting if an exact quadrature with the quadratrix were possible nonetheless. With this in mind, we will also discuss more liberal uses of the quadratrix than we allowed (but still without taking limits), and whether the circle can be squared with them (Conjecture \ref{ConELhat}).

\section{From rectification to quadrature}\label{Quadrature}

Before introducing the quadratrix, let us explain how to square the circle once its circumference has been rectified, i.e., a straight line segment of the same length has been produced. The constructions involved will also be helpful in translating geometry into algebra in Section \ref{Numb}.

In his Measurement of the Circle, Archimedes proved that a circle's area is equal to the area of the right triangle with the base equal to its circumference and the height equal to its radius, a geometric equivalent of our modern formula \cite[1.6]{Crippa}. Complementing this triangle so as to form a rectangle and cutting it in half, we get a rectangle of the same area, see Figure \ref{Triangcirc}. If $a$ and $b$ are the sides of the rectangle, to find a square of the same area one has to solve $x^2=ab$ geometrically, i.e., to construct a segment of length $x$. The equation is equivalent to the proportion $a:x=x:b$, and ancient Greek geometers called $x$ the mean proportional between $a$ and $b$. The next proposition shows how it can be constructed with straightedge and compass. The square with the rectangle's area is then constructed in Figure \ref{MeanProp}\,(b).
\begin{figure}[!ht]
\vspace{-0.1in}
\begin{centering}
(a) \includegraphics[scale=.25]{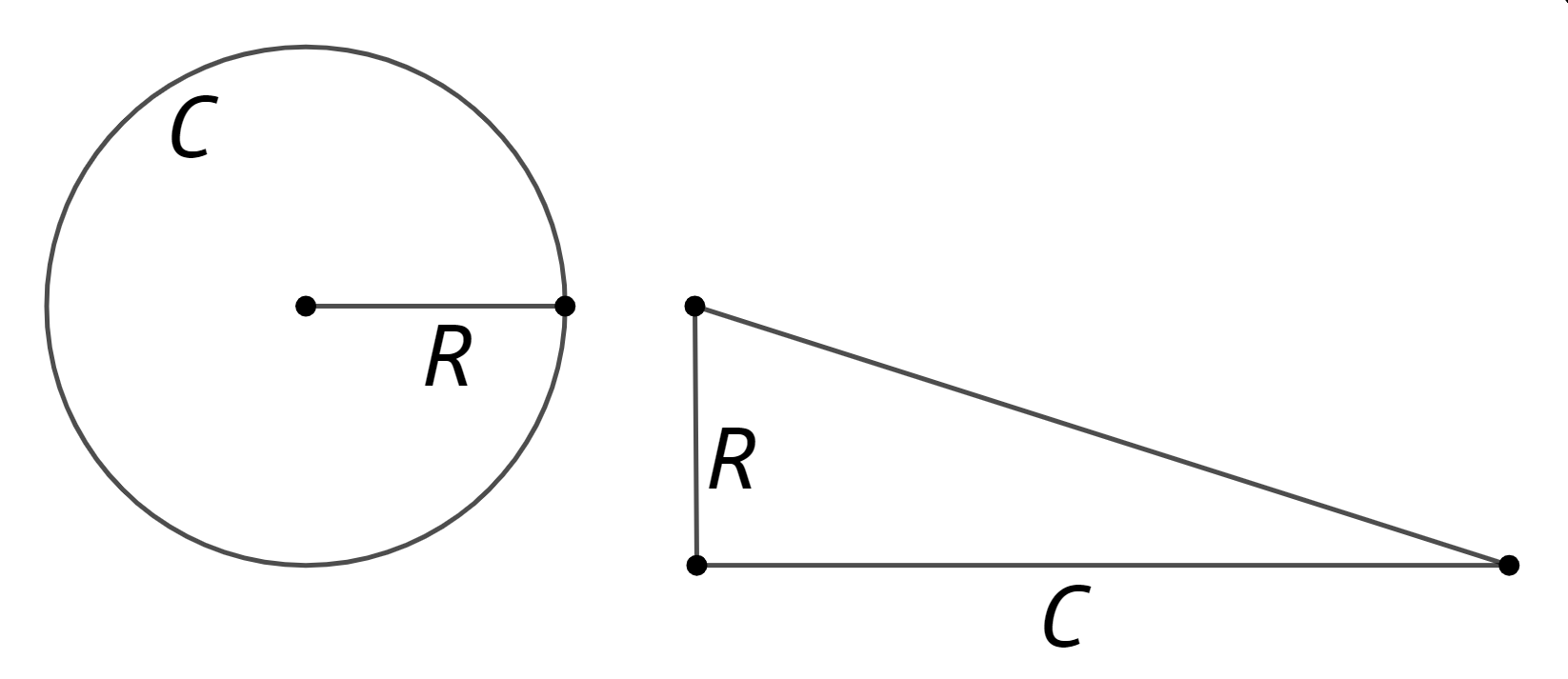}
\hspace{0.1in} 
(b) \includegraphics[scale=.25]{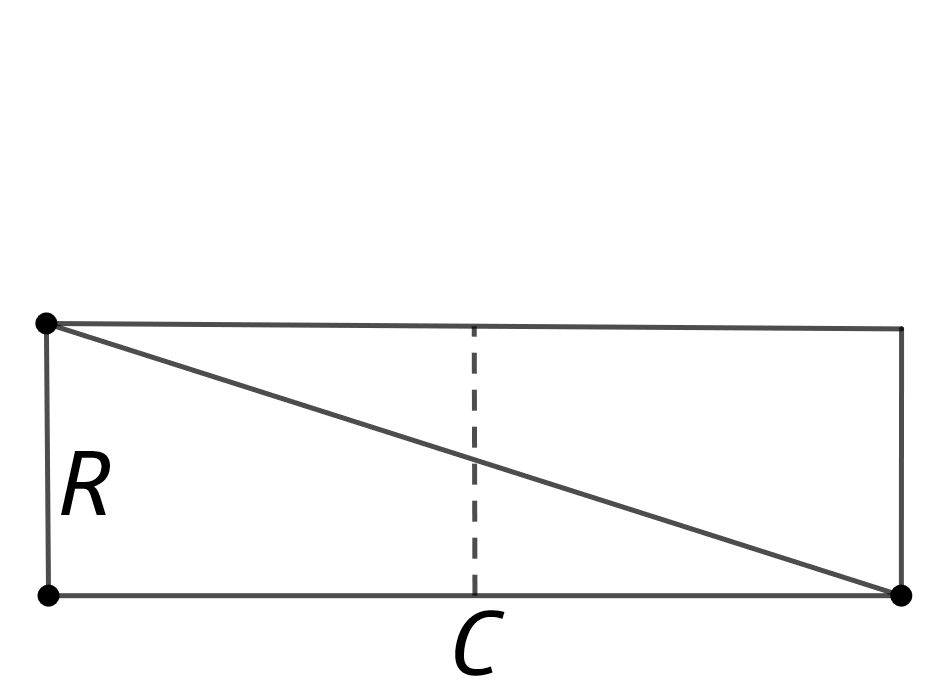}
\hspace{0.1in}
\par\end{centering}
\vspace{-0.2in}
\hspace*{-0.1in}\caption{\label{Triangcirc} (a)  A triangle equal in area to the circle; (b) constructing a rectangle equal in area to a triangle.}
\end{figure}
\begin{proposition}[\textbf{Inserting the mean proportional}]\label{InsMean} Given segments of lengths $a$ and $b$, one can construct a segment of length $x$ satisfying the proportion $a:x=x:b$ with straightedge and compass. 
\end{proposition} 
\begin{proof} The construction is shown in Figure \ref{MeanProp}\,(a). We place the segments on a straight line next to each other, then divide the merged segment in half, and draw a semicircle on it as the diameter. The perpendicular erected on the diameter from their common point $D$ up to its intersection with the semicircle is the segment we are looking for. 

Indeed, since $\angle{ACB}$ is an angle inscribed into a semicircle, it is a right angle. Therefore, the triangles $ACD$ and $CBD$ are both right triangles and $\angle{ACD}=\angle{DBC}$ because they both complement $\angle{CAD}$ to the right angle. This means that they are similar triangles and $AD:CD=CD:BD$, or $a:x=x:b$ with $x$ being the length of $CD$. 
\end{proof}
\begin{figure}[!ht]
\vspace{-0.1in}
\begin{centering}
(a) \includegraphics[scale=.3]{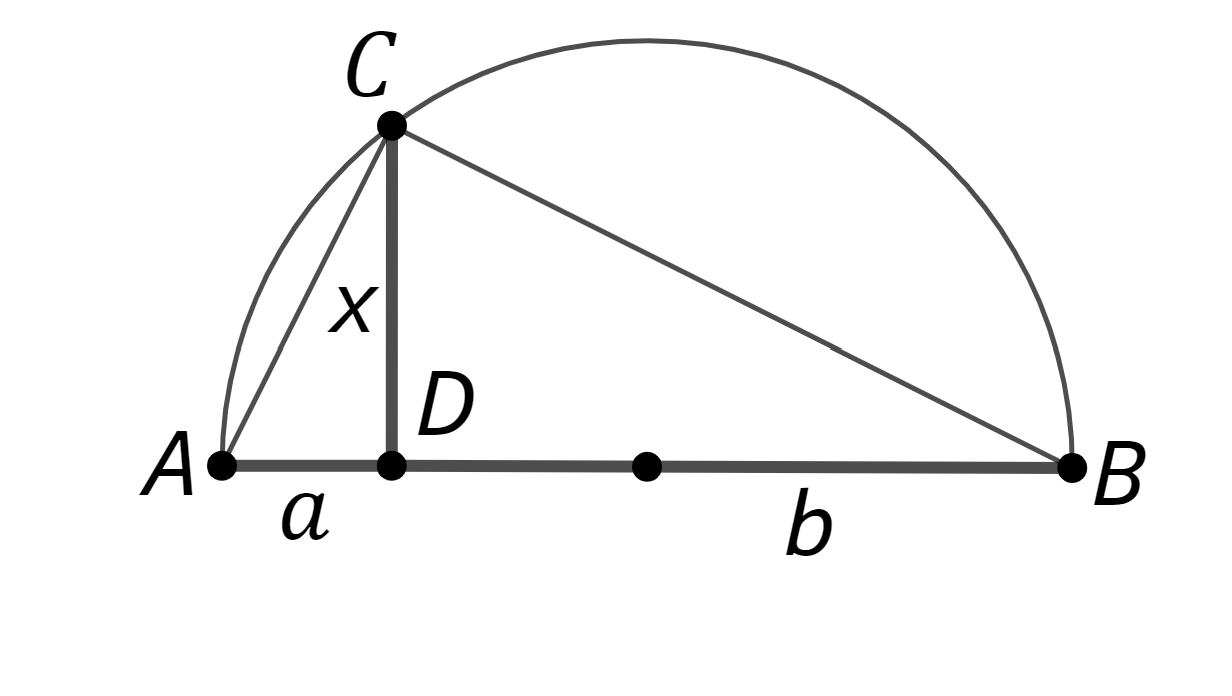}
\hspace{0.1in} 
(b) \includegraphics[scale=.3]{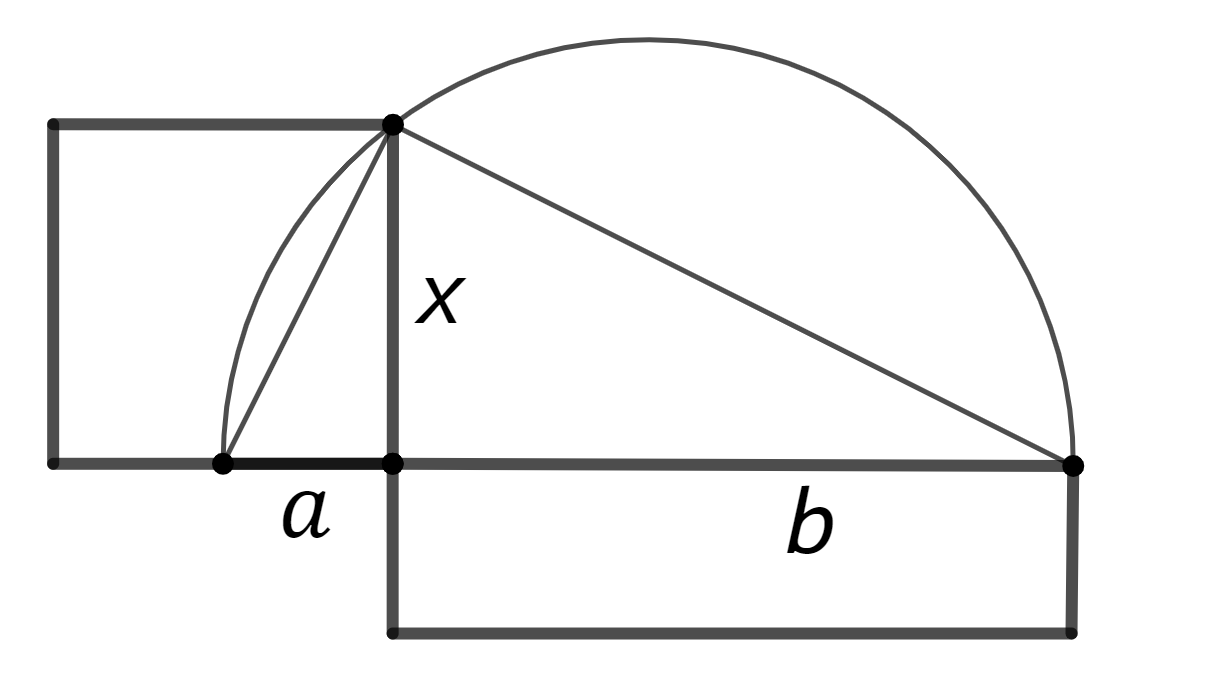}
\hspace{0.1in}
\par\end{centering}
\vspace{-0.2in}
\hspace*{-0.1in}\caption{\label{MeanProp} (a) Construction of the mean proportional; (b) squaring a rectangle by inserting the mean proportional.}
\end{figure}

Now suppose that we somehow managed to obtain a segment equal to the circumference of some specific circle, i.e., to rectify that circle. Then we can rectify (and, therefore, square) {\it any} circle with straightedge and compass. It is done by geometric scaling, what ancient Greeks called finding the fourth proportional. 
\begin{proposition}[\textbf{Finding the fourth proportional}]\label{Find4} Given segments of lengths $a$, $b$ and $c$, one can construct a segment of length $x$ satisfying the proportion 
$x:a = c:b$ with straightedge and compass.
\end{proposition}
\begin{proof}
When $b=c$, the construction is trivial as $x=a$. For the case $b<c$, the construction is shown in Figure \ref{fourthprop}\,(a). Erect the perpendicular $AA'$ to $AD$  of length $a$, and then another perpendicular $A'G'$ of length $c$ to $AA'$.
Choose $G$ on $AD$ so that the segment $AG$ has length $b$. Construct the line $GG'$ and extend it until it intersects the line $AA'$ with $O$ as the intersection point. Then construct the line $OD$ and extend it to the intersection with $A'G'$ with $D'$ as the intersection point. Then $A'D'$ is the fourth proportional of length $x$.

Indeed, triangles $OAG$ and $OA'G'$ both have a right angle, and share $\angle AOG$ by construction. Therefore, their third angles are also equal. By the same reasoning, $OAD$ and $OA'D'$ also have three equal angles. Hence, triangle $OAG$ is similar to $OA'G'$ and triangle $OAD$ is similar to $OA'D'$. Thus,  $A'D':AD = OA':OA = A'G':AG$, or $x:a = c:b$. The case $b>c$ is analogous.
\end{proof}
\begin{figure}[!ht]
\vspace{-0.1in}
\begin{centering}
\ \ (a)\ \includegraphics[scale=.262]{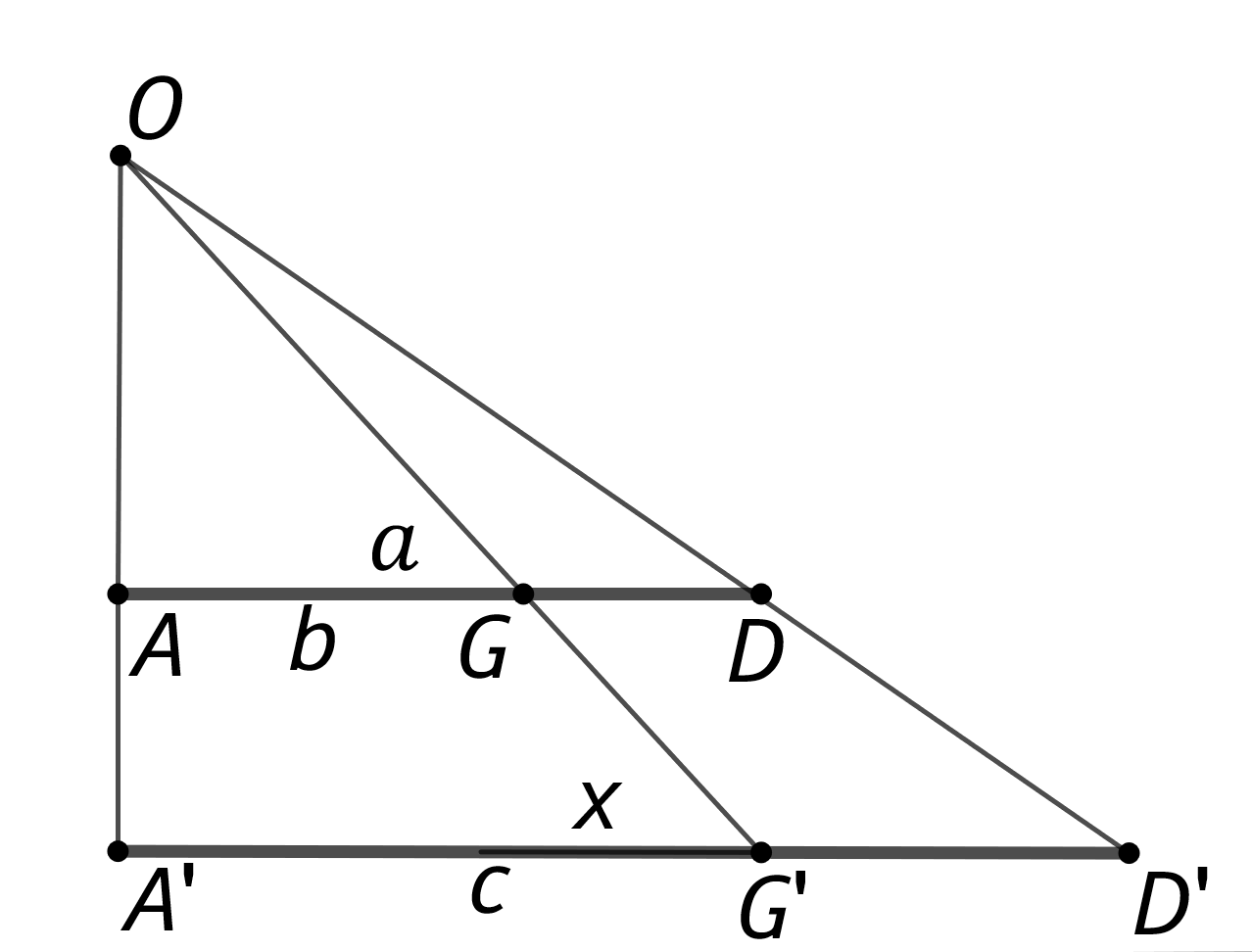}
\hspace{4em} (b)\ \includegraphics[scale=0.113]{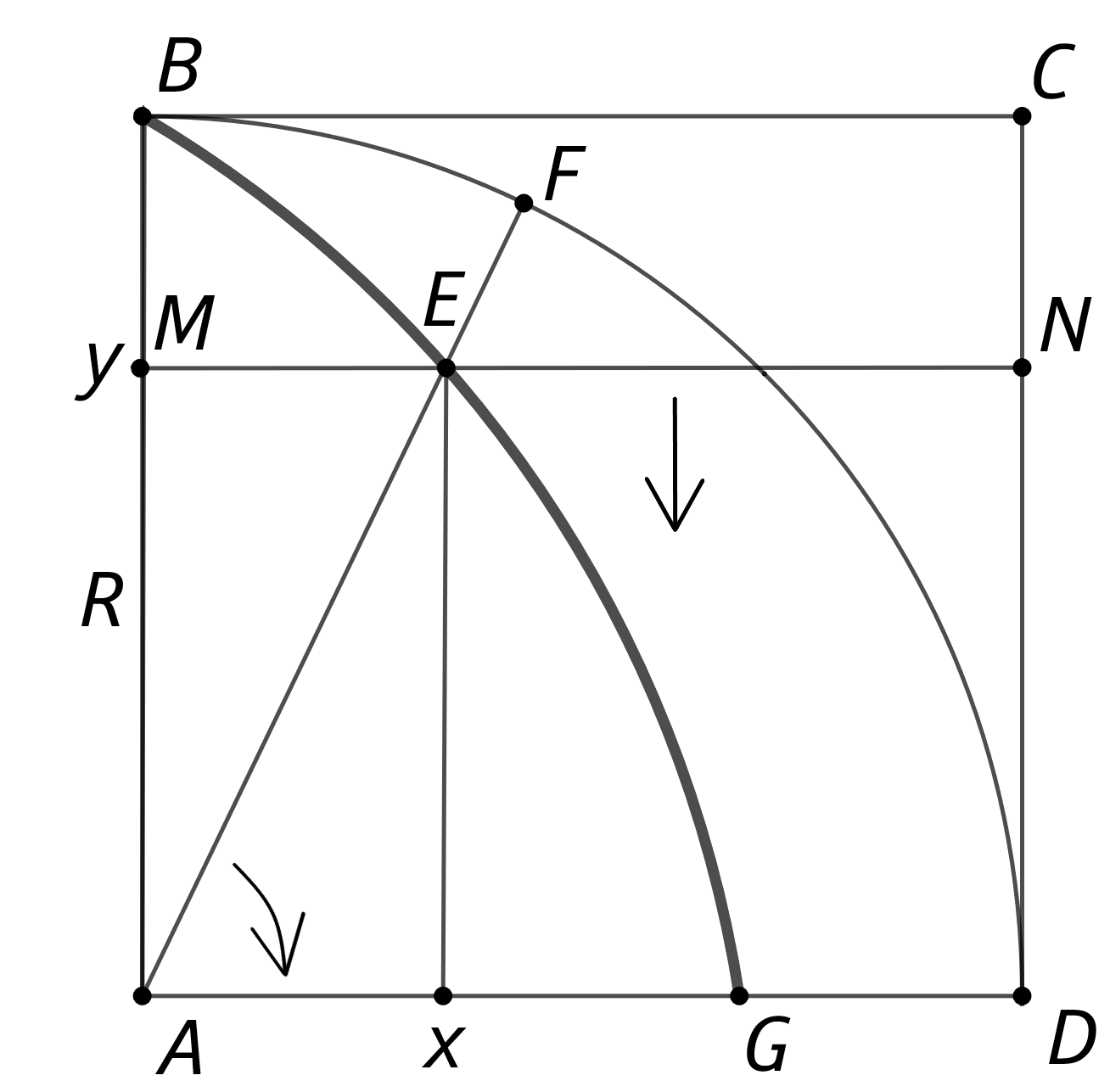}
\par\end{centering}
\vspace{-0.1in}
\hspace*{-0.1in}\caption{\label{fourthprop} (a) Construction of the fourth proportional; (b) genesis of the quadratrix.}
\end{figure}

Given circles with the radii $R_1$, $R_2$ and the rectified circumference of the first, $C_1$, the fourth proportional will be exactly the rectified circumference of the second, $C_2$. Thus, if we can construct two line segments that are to each other as the circumference of some circle is to its radius then we can square any circle. It is to the former task that the quadratrix was applied since antiquity.

\section{Rectification with the quadratrix}\label{Quadratrix}

The quadratrix, translated as ``square maker," is traced by the intersection point of two line segments in uniform motion, one linear, the other circular. According to the traditional story, which follows Proclus, Hippias of Elis introduced it originally c.\,420 BC to trisect an angle, and its original name was trisectrix. It is only c.\,350 BC that Dinostratus used it to square the circle. Little is known about Dinostratus, other than that he was the brother of Menaechmus who introduced conic sections into Greek geometry. Pappus confirms that Dinostratus (and Nicomedes) used the quadratrix for the quadrature, but the traditional story is controversial among historians \cite[II.5.4]{Sefrin}. 

To generate the quadratrix, begin with a square and move its top side uniformly down while rotating the left vertical side uniformly clockwise. Synchronize the motions so that both reach the bottom side at the same time, see Figure \ref{fourthprop}\,(b). The characteristic property of the quadratrix, implied by the uniformity of the generating motions, is that the leftmost point $M$ of the descending segment divides $AB$ in the same ratio as the rightmost point $F$ of the rotating segment divides the arc $BFD$, and hence also the right angle at the bottom left corner of the square.  

Denoting $R:=|AB|$, $x:=|ME|$, $y:=|AM|$, $\theta:=\angle DAF$, and measuring angles in radians, we have $\frac{\pi}{2}:\theta=R:y$ and $y=x\tan\theta$.
Therefore, $R:x=\pi\tan\theta:2\theta$. 
Let $G$ be the `terminal' point, where the quadratrix intersects the bottom side of the square. Since $\ds{\tan\theta:\theta\to1:1}$ when $\theta\to0$ we have $R:|AG|=\pi:2$. Therefore, $4R:|AG|=2\pi:1$, and they are to each other as the circumference of a circle is to its radius. This means that the circle of radius $|AG|$ can be rectified with the quadratrix. As we know from the previous section, this allows us to square any circle using only straightedge and compass.

\section{Sporus's objections}\label{Sporus}

Upon closer inspection, not all is well with the above rectification. Two objections to it were raised by Sporus, a philosopher approvingly quoted by Pappus \cite{Sefrin}. 

First, in order to synchronize the rotational motion of $AB$ with the linear motion of $BC$ in Figure \ref{fourthprop}\,(b), the ratio of the arc $BD$ to $AB$ needs to be known. Indeed, the speeds of the motions must be in this ratio for the descending and the rotating segments to arrive at the bottom of the square simultaneously. But this ratio is $\pi:2$, and it is the very ratio that we used the quadratrix to produce. Thus, Sporus objected, this `solution,' as presented, is logically circular.      
\begin{figure}[!ht]
\vspace{-0.1in}
\begin{centering}
\ \ \ (a)\,\includegraphics[scale=.11]{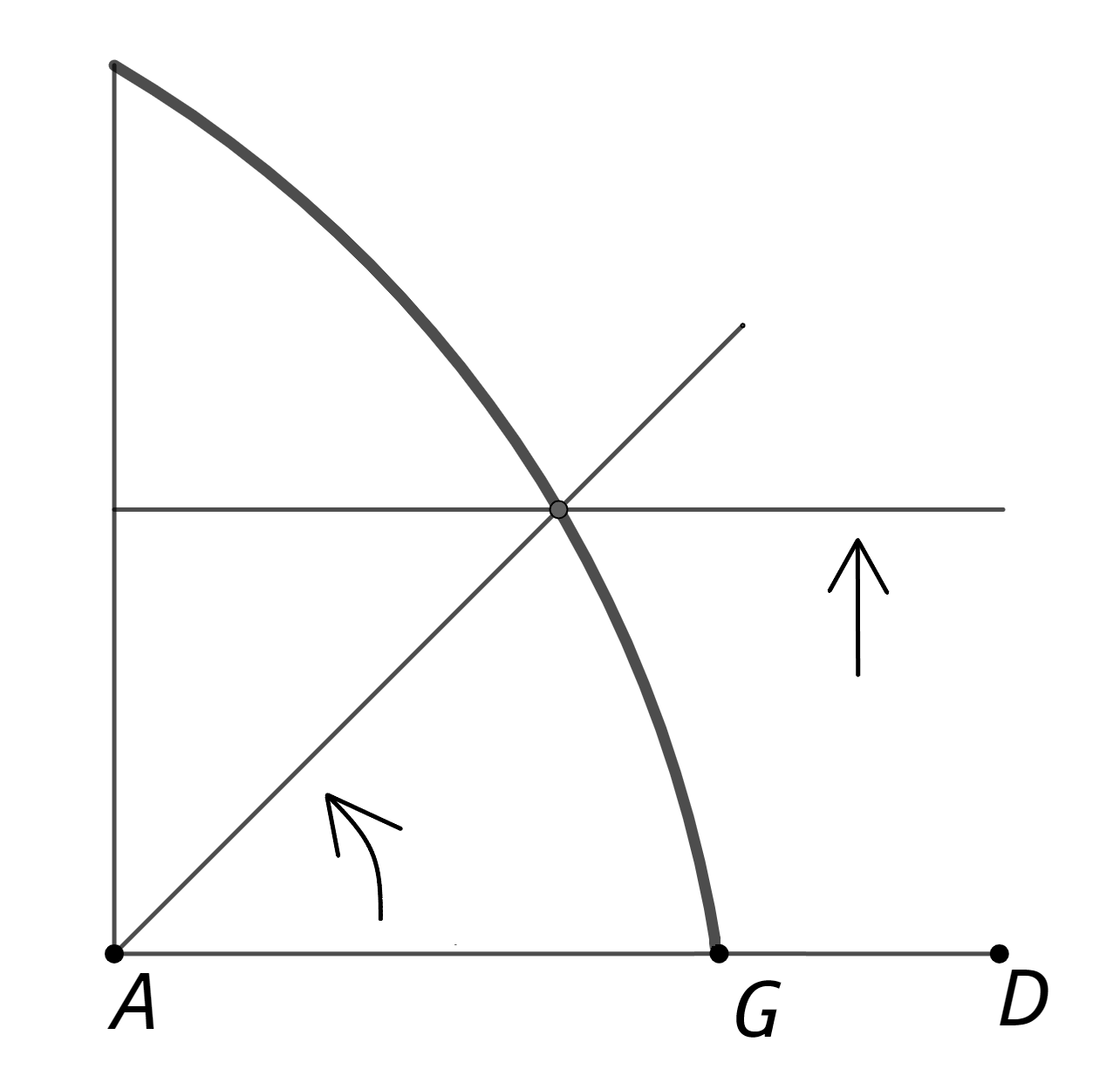}
\hspace{7em} (b)\,\includegraphics[scale=.22]{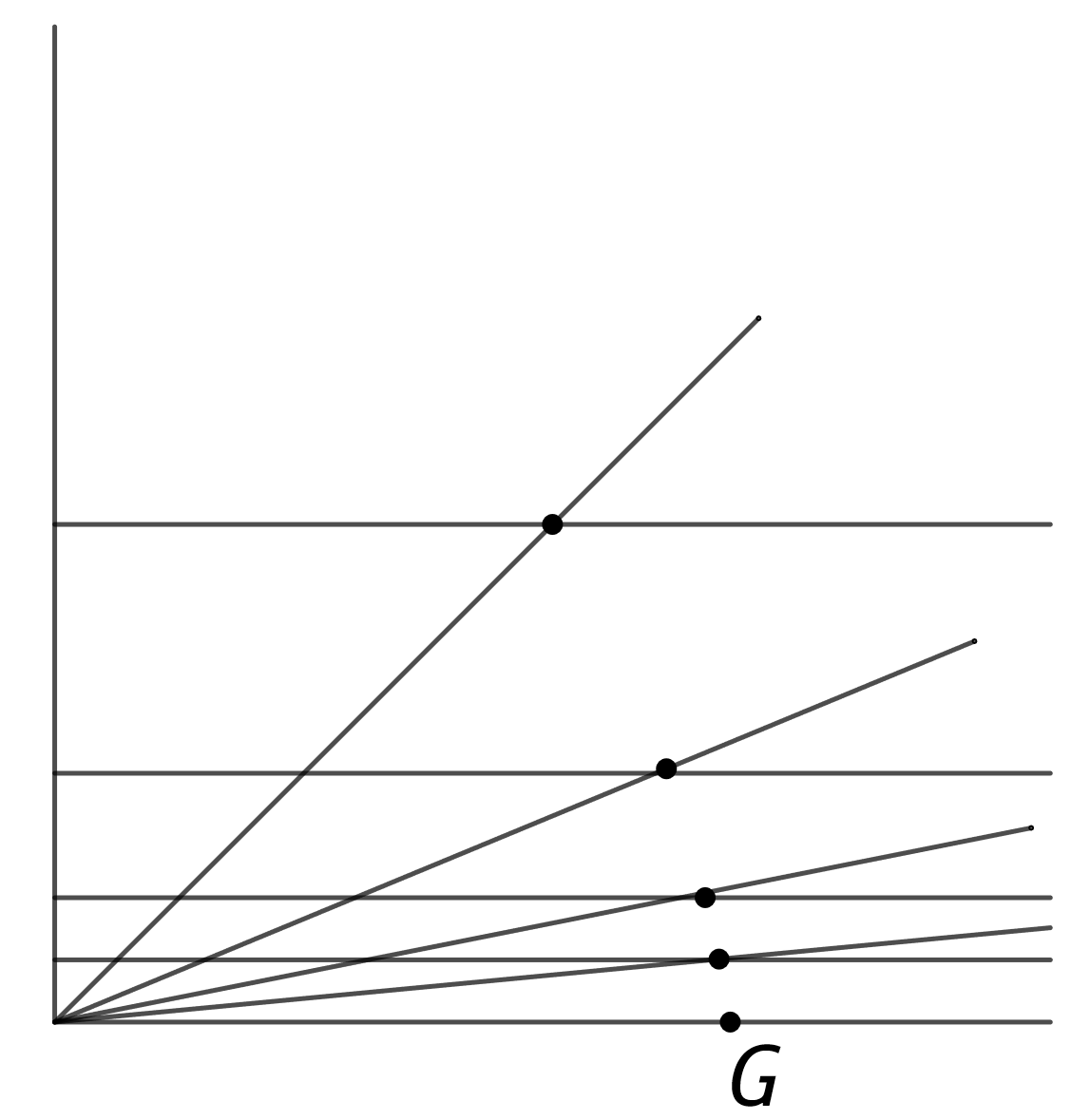}
\par\end{centering}
\vspace{-0.1in}
\hspace*{-0.1in}\caption{\label{quadbypts} (a) Bos's regenesis of the quadratrix; (b) points on the quadratrix constructed by repeated bisection that converge to the limit point.}
\end{figure}

The good news is that Pappus repaired it already in antiquity \cite{Sefrin}. He used 3D constructions, but in the 20th century, a Dutch historian of mathematics Henk Bos gave a much simpler fix \cite{Bos}. In Bos's regenesis, the pre-drawn square, which required synchronizing the motions, is not used, and the  directions of both motions are reversed, see Figure \ref{quadbypts}\,(a). From the same horizontal starting position, one line moves uniformly up, and the other uniformly rotates counterclockwise. Their intersection point still traces a quadratrix, which eventually intersects the vertical line. 

The bad news is that the point $G$ on $AD$ that we need for rectification is still not produced by this process. The ascending and the rotating lines coincide at the starting position, so there is no intersection point. The best we can do is to find $G$ by taking a limit along the quadratrix. Sporus's second objection was to this taking of the limit.

Indeed, the task was to give an {\it exact} construction of a square with the same area as the circle, not its approximation as a limit. Allowing approximations trivializes it. After all, we can inscribe regular polygons with $2^n$ sides into the circle and square them, all with straightedge and compass. We can even construct a sequence of points on the quadratrix converging to the limit point $G$ with straightedge and compass. Such a construction was proposed in 1604 by a Renaissance mathematician Christopher Clavius, who mistakenly thought that it produces every point on the quadratrix \cite{Mancosu}. Just bisect the right angle and the vertical segment successively as in Figure \ref{quadbypts}\,(b). The corresponding intersection points form a sequence that converges to the limit point $G$. Why bother with motions and curves at all?

The worse news is that the use of limits is not specific to the quadratrix, it pops up in solutions with other mechanical curves. For example, consider another curve generated by combining uniform linear and circular motions, the Archimedean spiral. Pappus credits Conon of Samos for its discovery c.\,245 BC \cite{Knorr}. Conon was a mathematician and astronomer who  became friends with Archimedes and shared his work with him. The curve was named after Archimedes because of extensive use of it in his book {\it On Spirals}, c.\,225 BC. Its shape might have been inspired by the water pumping screw, now also called the screw of Archimedes, which was in use at the time and which he perfected. 

Archimedes defines the spiral in his book as the path of a point moving uniformly along a straight line, which, at the same time, uniformly rotates about a fixed point \cite{Knorr}.
\begin{figure}[!ht]
\vspace{-0.1in}
\begin{centering}
\ \ \ (a)\,\includegraphics[scale=.26]{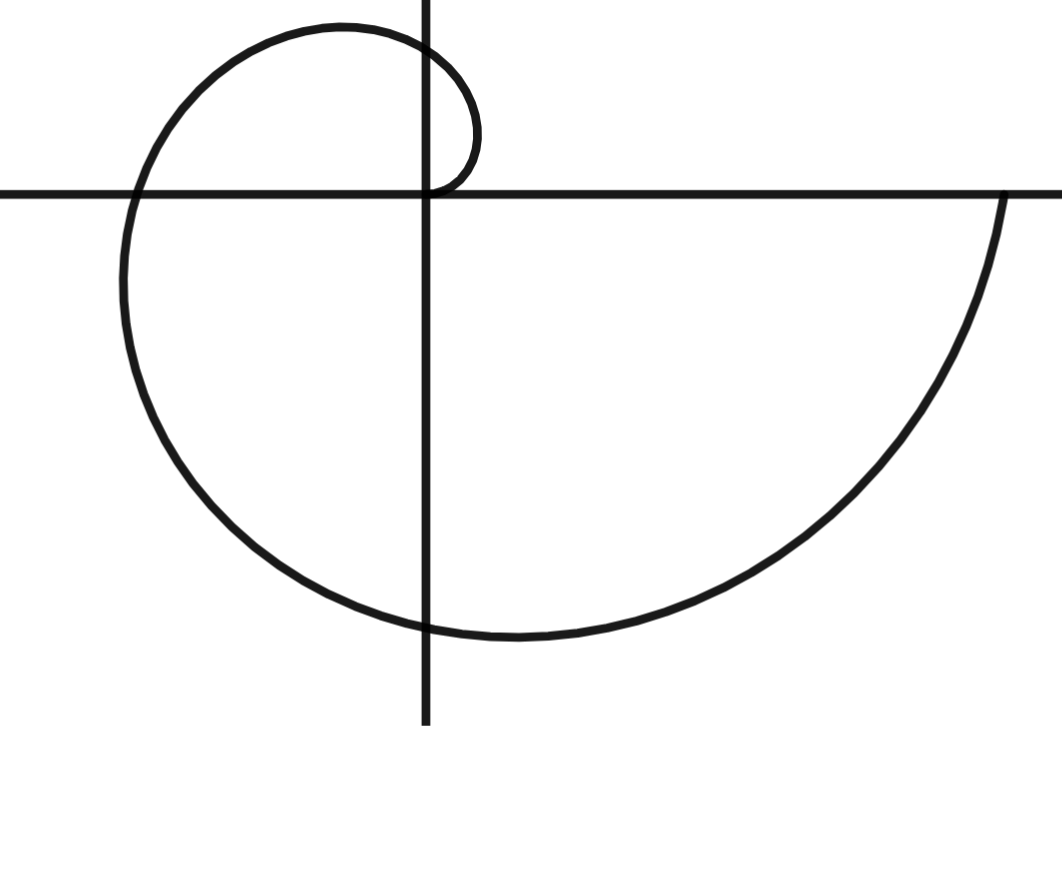}
\hspace{6em} (b)\,\includegraphics[scale=.11]{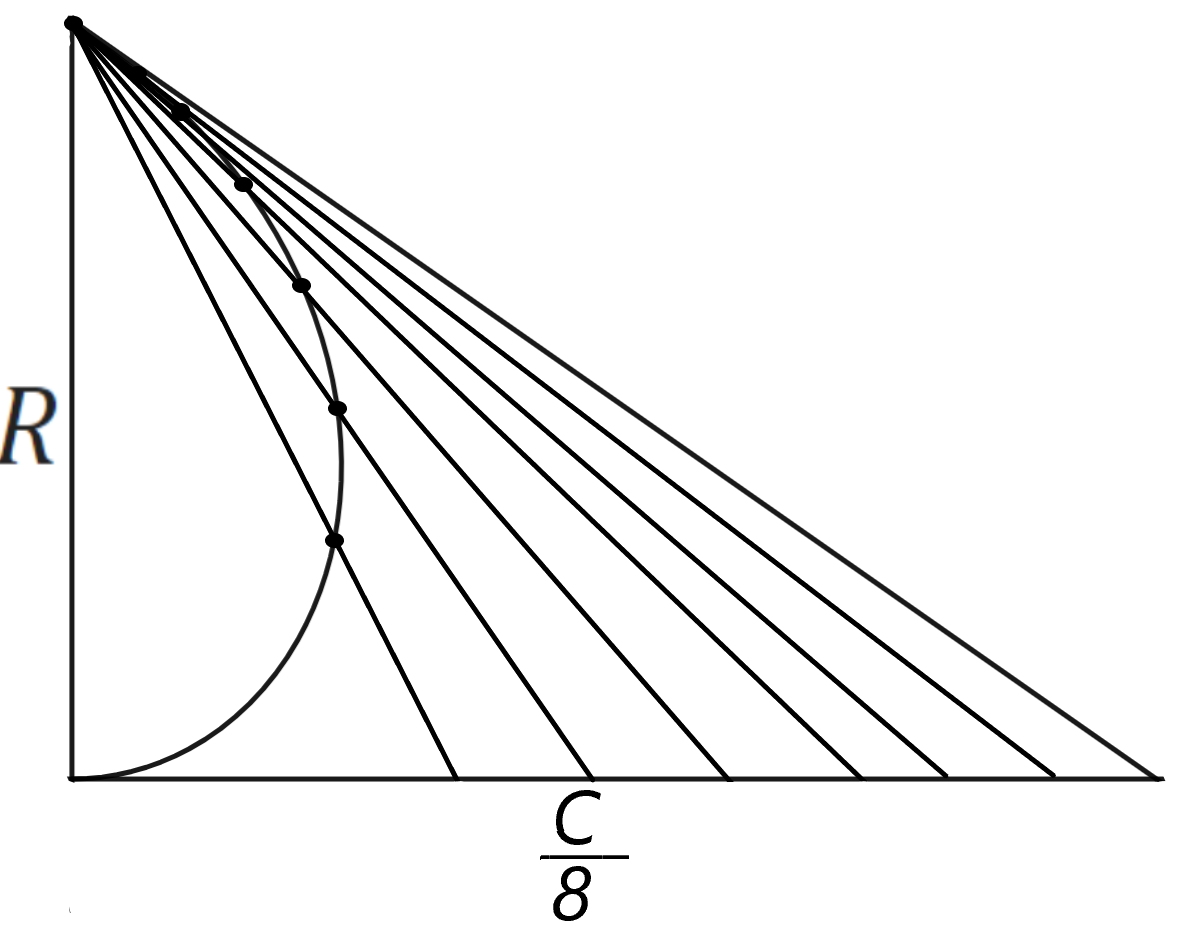}
\par\end{centering}
\vspace{-0.15in}
\hspace*{-0.1in}\caption{\label{archspiral} (a) Archimedian spiral after one full rotation; (b) rectifying the circle with the Archimedian spiral.}
\end{figure}
A single turn of the Archimedean spiral is shown on Figure \ref{archspiral}\,(a). Archimedes himself wisely refrained from using it for the quadrature, but his successors repurposed one of his theorems to do just that. Let $R$ be the distance from the quarter-turn point on the spiral to the initial point. The theorem in question states that the tangent to the spiral at the quarter-turn point will cut a segment on the initial tangent equal to one eighth of the circumference of the circle of radius $R$, see Figure \ref{archspiral}\,(b). Thereby, we get a triangle with the area equal to a rational multiple of the area of the circle, and some straightedge and compass fiddling will then square the circle.

However, even if the spiral is already drawn there is no (known) way to construct a tangent to it with straightedge and compass. The intuitive idea of fitting a straightedge to the curve until they `touch' is not an exact construction, although it may work well enough in practice. And when we try to formalize it, we are back to approximating a limit. We can draw a secant line to the spiral through the point of tangency and some nearby point, see Figure \ref{archspiral}\,(b). As the nearby point approaches the point of tangency, the secants will approach the tangent. But this is just a variation on approximating the limit point on the quadratrix. 

At this point, one may start to suspect that a limit process is an unavoidable feature of squaring the circle with mechanical curves like the quadratrix or the Archimedian spiral. If that is, indeed, the case then the answer to the title question of this paper would be negative. No cheating---no quadrature. Adding mechanical curves to straightedge and compass does not allow squaring the circle in the exact sense that ancient Greeks intended. But how do we prove it?

\section{General anglesection}\label{Anglesec}

The lesson of Sporus's second objection is that we may not use the limit point on the quadratrix. But what can we still construct without it? As already mentioned, the quadratrix was originally used to trisect an angle, i.e., divide it into three equal parts. The limit point is not needed for that. Pappus even showed that one can divide any given angle in any given ratio (of two segments) without it as well \cite{Bos}.

\begin{figure}[!ht]
\vspace{-0.1in}
\begin{centering}
\ \ \ (a)\,\includegraphics[scale=.23]{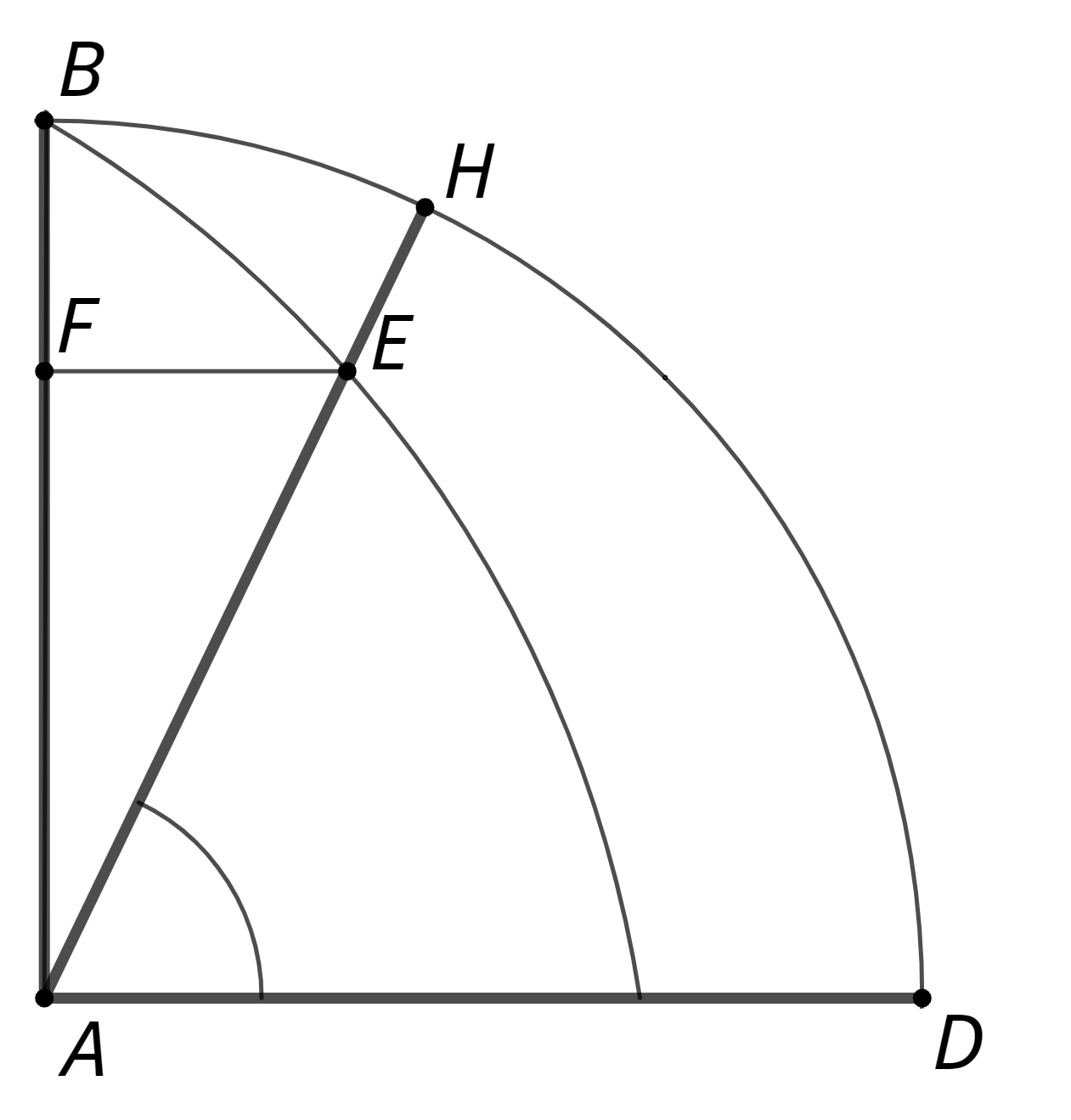}
\hspace{6em} (b)\,\includegraphics[scale=.23]{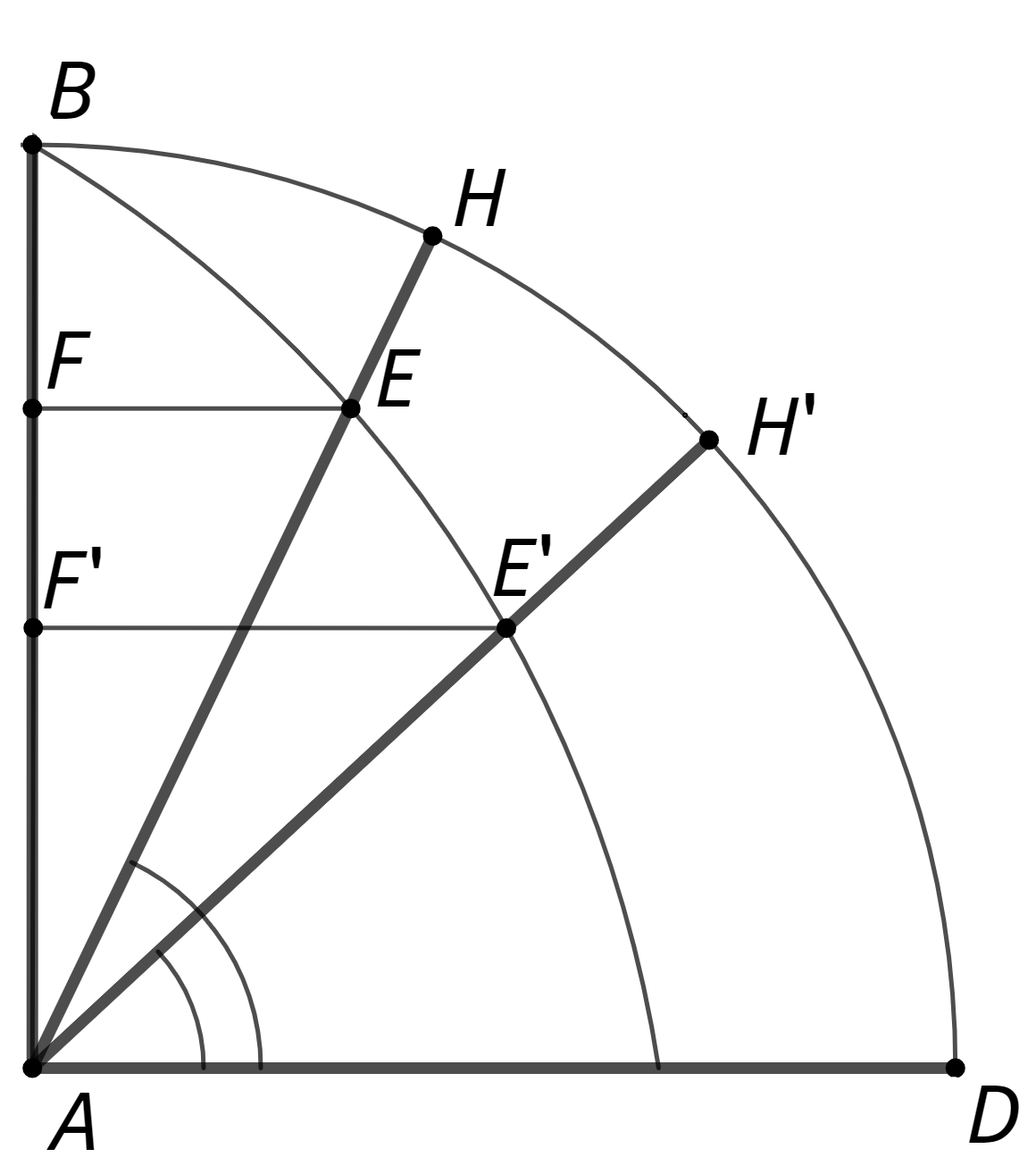}
\par\end{centering}
\vspace{-0.15in}
\hspace*{-0.1in}\caption{\label{anglesectorlbl} (a) Right anglesection with the quadratrix: $\angle DAH$ divides the right angle in the same ratio as $F$ divides $AB$; (b) acute anglesection with the quadratrix: $\angle DAH'$ divides $\angle DAH$ in the same ratio as $F'$ divides $AF$.}
\end{figure}

Let us divide the right angle first, see Figure \ref{anglesectorlbl}\,(a). Given a segment $AB$ cut in some ratio by $F$, construct the quadratrix on segments $AB$ and $AD$ equal to $AB$. Erect the perpendicular to $AB$ at $F$, and let $E$ be its intersection point with the quadratrix. Draw the line $AE$ until it intersects the circular arc $BD$ at $H$. It follows from the uniformity of linear and circular motions that $H$ cuts the arc $BD$, and hence the right angle $\angle DAB$, in the same ratio as $F$ cuts the segment $AB$. 

Now let us divide some acute angle $\angle DAH$ in a given ratio, see Figure \ref{anglesectorlbl}\,(b). First, we can divide $AB$ in the same ratio as $\angle DAH$ divides the right angle $\angle DAB$ by simply repeating the above construction in reverse. Namely, draw the line $AH$ and drop the perpendicular to $AB$ from its intersection point $E$ with the quadratrix to get $F$. Then, divide $AF$ by $F'$ in the given ratio, and, finally, divide the right angle in the ratio $AF':AB$ to get $H'$. Thus, $\angle DAH':\angle DAH=AF':AF$ as desired.

Since any angle is the sum of an acute angle with an integer multiple of the right angle, and we can divide both in a given ratio, we can divide any given angle in that ratio. This is the {\it general anglesection} and no part of it requires using the limit point on the quadratrix. Pappus showed that the general anglesection is possible with the Archimedean spiral as well and does not require drawing a tangent to it either. Bos even suggested that it ``seems to be the only rationale" for combining motions as in the quadratrix and the Archimedean spiral \cite{Bos}. We will be more generous and take both the right anglesection and its reverse separately as quadratrix's legitimate uses. Can we square the circle with them only? 

\section{From geometry to algebra}\label{Numb}

Historically, questions about the possibility or impossibility of geometric constructions were only answered when they were converted into algebraic questions about numbers. To see the correspondence, pick a segment in the plane and declare it a unit segment, i.e., having length $1$. Take the line through it as the horizontal axis, one of its endpoints as the origin, and set up Cartesian coordinates in the plane using it as the scale. 
\begin{definition} A real number is called constructible when its absolute value is the length of a segment that can be constructed starting from the unit segment with straightedge and compass. A point is called constructible when both its Cartesian coordinates are constructible. 
\end{definition}
Ancient Greeks resisted such non-intrinsic assignment of numbers to geometric objects, so the idea had to wait until Descartes in the 17th century. He answered, in part, which numbers are constructible with his algebra of segments \cite[3.13]{Hart}, a collection of geometric constructions on segments that correspond to numerical operations on their lengths. The algebra of segments shows that constructible points answer their name---they really are constructible with straightedge and compass.  

However, historians showed that Descartes did not fully appreciate the iterative nature of straightedge and compass constructions in his impossibility arguments  \cite{Lutzen}. For example, he thought that angle trisection, which leads to a cubic equation with three roots, is impossible with straightedge and compass simply because a single intersection of lines and circles produces at most two points. That one can intersect them repeatedly, thereby producing roots of equations of arbitrarily high degree, was only taken into account by Gauss in his 1801 work on inscribing regular polygons into the circle. Below, we sketch an argument for the translation of straightedge and compass constructions into algebraic operations. For a more detailed and formal presentation of the induction involved, the reader is referred to \cite{Hart,JMP,Waerden}.

Curiously, Gauss thought so little of impossibility results that he mentioned only in passing and without proof that a $9$-gon cannot be inscribed with straightedge and compass, which implied the impossibility of the trisection. It was only proved by Wantzel in 1837 following Abel's reframing of impossibility results and based on Gauss's algebraic translation of straightedge and compass constructions  \cite{LutzenW}. 
\begin{theorem}[Descartes, Gauss]\label{DesGau} A real number is constructible if and only if it can be obtained from $1$ by applying a finite number of algebraic operations and takings of $\sqrt{x}$ for previously constructed
positive numbers $x$.
\end{theorem}
\begin{proof}
Given the unit segment and segments of lengths $a,b$, segments with lengths $a+b$, $a-b$ can be constructed by laying them out on the same line; $a\cdot b$, $a/b$ by taking the fourth proportionals in $x:b=a:1$ and $x:a=1:b$, respectively; and $\sqrt{a}$  by inserting the mean proportional in $1:x=x:a$. By Propositions \ref{InsMean} and \ref{Find4}, all of this is constructible with straightedge and compass. Therefore, all algebraic operations and takings of square roots can be performed with them. 

For the converse, note that constructing points with straightedge and compass reduces to producing lines and circles and intersecting them with each other. The coefficients of their equations are rational functions of coordinates of already constructed points and squared distances between them (for circles' radii). Hence, the coordinates of the intersection points can be found by solving linear and/or quadratic equations with those coefficients, i.e., by applying field operations to them and square roots to their positive combinations.
\end{proof}
Since we are dealing with anglesection we must consider constructible angles in addition to lengths and points. Rather than treating them separately, we will simply identify them with constructible points on the unit circle. Indeed, given an angle, we can lay off an equal one from the horizontal axis at the origin with straightedge and compass. Intersecting its inclined side with the unit circle produces the corresponding point on it. Conversely, any angle can be produced by drawing the line through a point on the unit circle and the origin. This way, dividing angles reduces to dividing arcs of the unit circle.
\begin{definition} Let the right anglesector (RA) be the tool that divides the right angle in the same ratio as that of any two given segments, and the reverse right anglesector (RRA) be the tool that divides segments in the same ratio as a given acute angle divides the right angle. A number is called RA-constructible or RRA-constructible when its absolute value is the length of a segment that can be constructed starting from the unit segment with straightedge, compass, and the corresponding tool. The same terms are applied to points when their coordinates are constructible accordingly. 
\end{definition}
\noindent We will now extend the Descartes-Gauss theorem to incorporate these anglesector tools. As should be expected, the functions representing their use are no longer algebraic in the traditional sense of algebra. This contrasts with extensions by marked straightedge in \cite{Baragar}, and by angle trisector in \cite{Gleason}. 
\begin{theorem}\label{AnalAngsec} A real number is RA/RRA-constructible if and only if it can be obtained from $1$ by applying a finite number of algebraic operations, and takings (for previously constructed
numbers $x$) of $\sqrt{x}$ with positive $x$ and of 
\smallskip

\noindent \textup{(RA)} $\sin(\pi x)$ with real $x$; 

\noindent \textup{(RRA)} $\frac1\pi\arcsin(x)$ with real $x$, $|x|\leq1$.
\end{theorem}
\begin{figure}[!ht]
\vspace{-0.1in}
\begin{centering}
\ \ \ (a)\,\includegraphics[scale=.125]{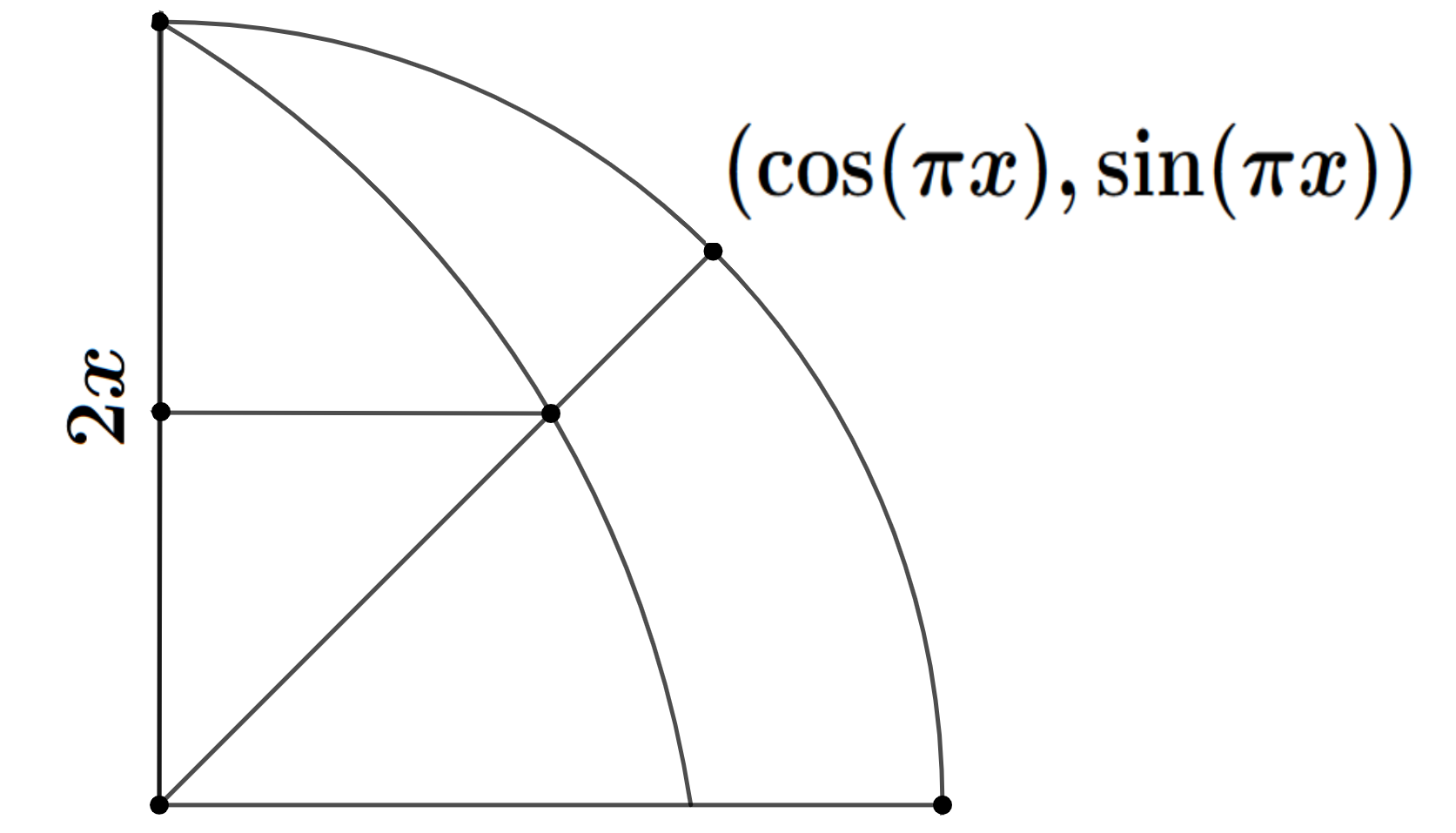}
\hspace{0em} (b)\,\includegraphics[scale=.125]{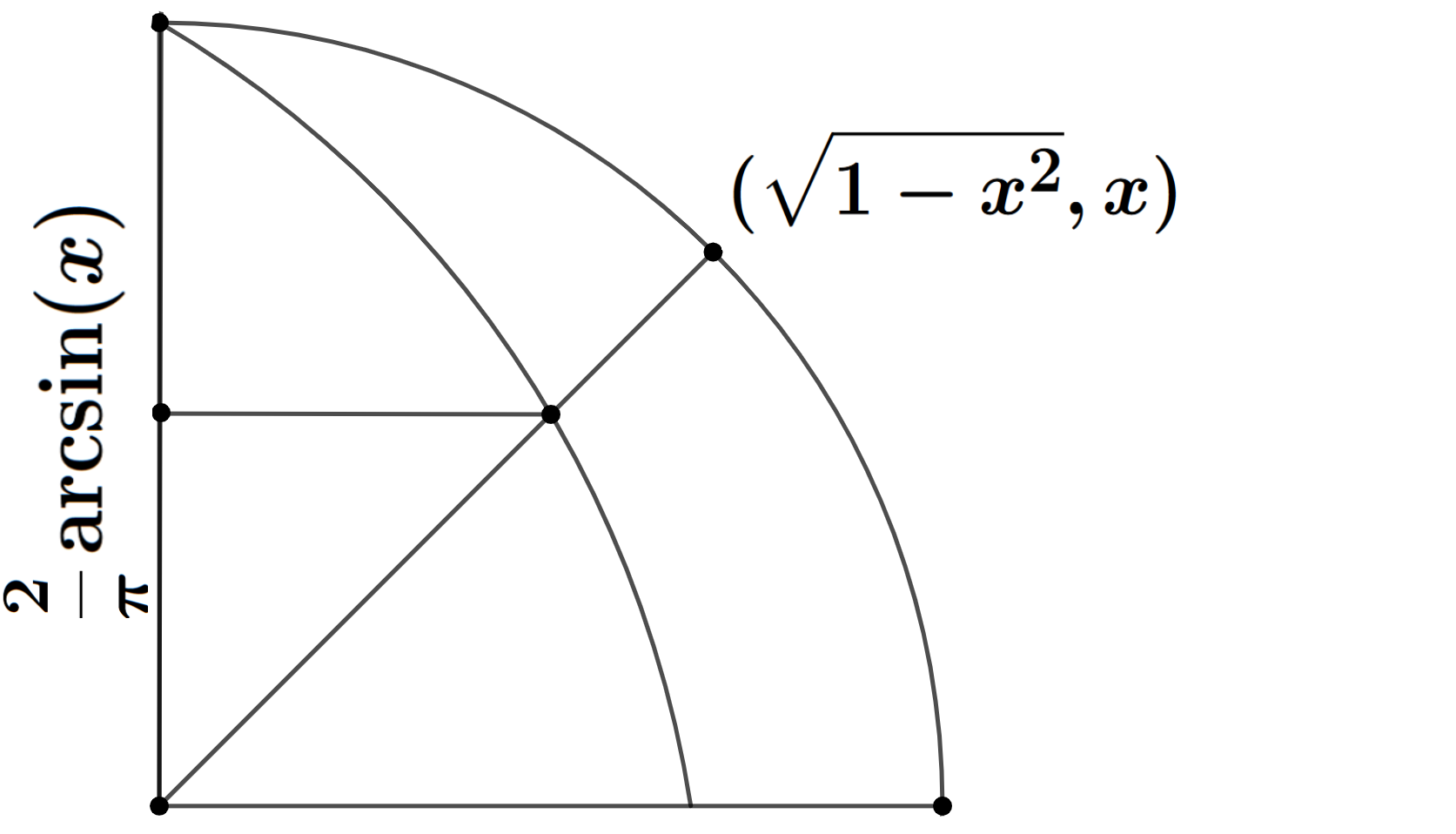}
\par\end{centering}
\vspace{-0.15in}
\hspace*{-0.1in}\caption{\label{anglesectorplain}(a) The point obtained by dividing the right angle in the ratio $2x:1$;\\ (b) the point obtained by dividing a unit segment in the same ratio that an acute angle divides the right angle.}
\end{figure}
\begin{proof} Due to Theorem \ref{DesGau}, we only need to find functional counterparts for the anglesectors.
\smallskip

\noindent \textup{(RA)} Given $x$, we can construct such $x'$ that $0\leq x'\leq\frac12$ and $\sin(\pi x)=\pm\sin(\pi x')$ with straightedge and compass. So assume that $x$ is already in that range. We then double it and divide the right angle in the ratio $2x:1$. The vertical coordinate of the corresponding point on the unit circle is then $\sin(\pi x)$, see Figure \ref{anglesectorplain}\,(a). Conversely, a point produced by dividing the right angle in a given ratio $2x:1$ has coordinates $\big(\!\cos(\pi x), \sin(\pi x)\big)$, and $\cos(\pi x)=\pm\sqrt{1-\sin^2(\pi x)}$, so its coordinates can be obtained by the listed operations. 
\smallskip

\noindent \textup{(RRA)} Given $x$, we  construct the point with coordinates $\big(\!\sqrt{1-x^2}, |x|\big)$ with straightedge and compass. The reverse anglesector then produces a vertical segment of length $\frac2\pi\arcsin|x|$, see Figure \ref{anglesectorplain}\,(b), and bisecting it produces what we want. Conversely, reverse anglesection converts a point with coordinates $\big(\!\sqrt{1-x^2}, x\big)$ for $0<x<1$ into one with coordinates $\big(0, \frac2\pi\arcsin(x)\big)$, and those can be obtained by the listed operations.
\end{proof}
Using Theorem \ref{AnalAngsec}, one can easily find the functional counterpart to the general anglesector as well. We leave it as an exercise for the reader.

\section{Anglesection and transcendental numbers}\label{Trans}

Now that we have a functional representation of constructions with anglesectors, geometric constructibility questions can be converted into parallel algebraic ones. To tackle them, it will be convenient to consider a larger class of numbers that includes constructible numbers and is easier to work with.
\begin{definition} A number is called algebraic when it is a root of a polynomial with integer coefficients, and transcendental otherwise. The set of algebraic numbers is denoted $\oQ$.
\end{definition}
It follows from the Descartes-Gauss theorem that any constructible number is algebraic. But, as Pierre Wantzel proved in 1837, the converse is false. The number $\sqrt[3]{2}$, needed to solve the cube duplication problem, is algebraic but not constructible. Even earlier, in 1824, Abel proved that the roots of $2x^5-10x+5$ and other quintic polynomials are not constructible either, and would remain so even if we allowed not just square roots but  radicals of any order.

One can call complex numbers constructible when they are represented by constructible points in the complex plane, and such numbers are still algebraic \cite[5.4]{Niven}, \cite[8.9]{Waerden}. But $e$ and $\pi$ are not. Hermite proved $e$ transcendental in 1873, and Lindemann proved $\pi$ the same in 1882. Since $\pi$ is not algebraic it is all the more not constructible. Therefore, the circle cannot be squared with straightedge and compass. Thereby, the 2,500 year old quadrature problem  was finally solved, or rather proved unsolvable.

But what if we add the anglesector tools introduced in Section \ref{Numb}? By Theorem \ref{AnalAngsec}, we can generate RA-constructible numbers recursively as follows. Start with the rational numbers $S_0:=\Q$. Then take all square roots $\sqrt{r}$ of positive rational numbers and all numbers $\sin(\pi r)$, and combine them with each other using the four field operations. We can denote the set of numbers so obtained by $\Q\big(\sqrt{\Q^+},\sin(\pi \Q)\big)$, with the parentheses standing for adjoining all algebraic combinations of the numbers listed. Then iterate the exercise. Similarly, for the RRA-constructible numbers we do the same with $\sin(\pi x)$ replaced by $\frac1\pi\arcsin(x)$ and applied only to previously generated numbers that fall on $[-1,1]$. Finally, we can combine the right anglesector and its reverse in our procedures. All in all, we get three towers of number fields.
\begin{definition}\label{SAtowers} Let $S_0:=\Q$ and recursively  $S_{n+1}:=S_n\big(\sqrt{S_n^+},\sin(\pi S_n)\big)$; $A_0:=\Q$ and  $A_{n+1}:=A_n\big(\sqrt{A_n^+},\frac1\pi\arcsin(A_n\cap[-1,1])\big)$; $\SA_0:=\Q$ and  
$$
\SA_{n+1}:=\SA_n\left(\sqrt{\SA_n^+},\sin(\pi \SA_n),\frac1\pi\arcsin(\SA_n\cap[-1,1])\right).
$$
We denote $S:=\bigcup_{n\geq0}S_n$, $A:=\bigcup_{n\geq0}A_n$, and $\SA:=\bigcup_{n\geq0}\SA_n$.
\end{definition}
\noindent By Theorem \ref{AnalAngsec}, we have that a real number $x$ is RA-constructible if and only if $x\in S$, and RRA-constructible if and only if $x\in A$.

It is easy enough to see that $\pi\not\in S_1$, i.e., the unit circle's circumference cannot be rectified with a single application of the right anglesector to rational segments. This is because $\sin(\pi r)$ is algebraic for any rational $r$  \cite{Lehmer}, \cite[5.4]{Niven}. As square roots are also algebraic, and algebraic numbers are closed under field operations (i.e., they form a field, just like $\Q$, $\R$,
and $\C$), we can conclude that $S_1\subseteq\oQ$. But we know that $\pi\not\in\oQ$ from Lindemann's result.

The reason $\sin(\pi r)$ are algebraic is that we can use trigonometric identities for multiple angles to generate a polynomial equation with integer coefficients that  they are roots of. For example, suppose $r = \frac{2}{5}$. By the identity for $\sin(5\theta)$,
$$
\sin(5\theta)=16\sin^{5}(\theta)-20\sin^{3}(\theta)+5\sin(\theta).
$$
In our case, $\theta = \frac{2\pi}{5}$, so $\sin(5\theta)=\sin(2\pi)=0$ and $x:=\sin(\theta)$ is a root of a polynomial of degree $5$, namely, $16x^{5}-20x^{3}+5x$. Analogously, as long as $\theta=\pi r$ is a rational multiple of $\pi$, we can use a multiple angle formula for $\sin(n\theta)$, with $n$ the denominator of $r$, to produce a polynomial equation that it solves. 

However, already in $S_2$ we encounter numbers like $\sin(\pi\sqrt{2})$. That they are transcendental follows from the following theorem proved independently by  Aleksandr Gelfond and Theodor Schneider in 1934 \cite{Lang71}, \cite[5.4]{Niven}.
\begin{theorem}[Gelfond, Schneider] Let $a,b$ be algebraic numbers with $a\neq0,1$ and $b$ irrational. Then $a^b$ is transcendental.
\end{theorem}
\noindent Here $a^b:=e^{b\ln a}$ with some choice of value for $\ln a$, the theorem holds for any such choice. In particular, taking $a=-1$, $\ln(-1)=i\pi$, and $b=\sqrt{2}$ we conclude that $(-1)^{\sqrt{2}}$ is transcendental. By Euler's formula \eqref{exp-1}, 
$$
(-1)^{\sqrt{2}}=e^{i\pi\sqrt{2}}=\cos(\pi\sqrt{2})+i\sin(\pi\sqrt{2}),
$$
and so $\sin(\pi\sqrt{2})$ must be transcendental as well. If it were algebraic then $\cos(\pi\sqrt{2})$ would also be algebraic, and hence so would be $(-1)^{\sqrt{2}}$, a contradiction. Thus, Lindemann's theorem does not tell us that $\pi\not\in S_2$.

It gets even worse for the reverse anglesector. Consider $\frac1\pi\arcsin(r)$ with a rational $r$. If it is algebraic then $r=\sin(\pi b)$ for some algebraic $b$. But we know from the Gelfond-Schneider theorem that for $b$ irrational $\sin(\pi b)$ is not even algebraic, let alone rational. And if $b$ is rational, Olmsted showed in 1945 with just elementary trigonometry that a rational $r=\sin(\pi b)$ can only be $r=0$, $\pm\frac12$ or $\pm1$ \cite[App.\,D]{Niven}, \cite{Olm}. For all other rational $r$, the number $\frac1\pi\arcsin(r)$ must be transcendental. So Lindemann's theorem does not even tell us that $\pi\not\in A_1$.

\section{Schanuel conjecture}\label{Schan}

To make headway, we need to distinguish $\pi$ not just from algebraic numbers, but from other transcendental numbers as well. The following definition introduces a concept suitable for that.
\begin{definition} Complex numbers $z_1,\dots,z_n$ are called algebraically independent over $\mathbb{Q}$ when there is no polynomial in $n$ variables with rational coefficients, not all of them $0$, such that $p\,(z_1,\dots,z_n)=0$. They are called linearly independent over $\mathbb{Q}$ when there is no such linear polynomial.
\end{definition}
\noindent Note that for a single number being ``algebraically independent" just amounts to being transcendental. Indeed, we can turn any polynomial with rational coefficients into a polynomial with integer coefficients having the same roots by simply multiplying all the coefficients by their least common denominator. 

In these terms, what we need is to prove algebraic independence of $\pi$ from $\sin(\pi\sqrt{2})$, and all other numbers added in the construction of $\SA$. Alas, the modern theory of transcendental numbers is still very far from proving such strong claims. However, as reported by Lang in 1966, his former PhD student  Stephen Schanuel made a conjecture that turned out to imply many of them \cite{Cheng,Chow,Lang71,Terzo}.
\begin{conjecture*}[Schanuel] If $z_1,\dots,z_n$ are complex numbers linearly independent over $\mathbb{Q}$ then among the numbers $z_1,\dots,z_n, e^{z_1},\dots,e^{z_n}$ at least half are algebraically independent over $\mathbb{Q}$.
\end{conjecture*}
\noindent The Schanuel conjecture is very strong. Many known difficult results and open conjectures follow from it. For example, it trivially implies Lindemann's theorem that $\pi$ is transcendental. Indeed, among $i\pi$ and $e^{i\pi}=-1$ there must be at least one algebraically independent (transcendental) number, so $i\pi$ is it. Since $i$ is algebraic $\pi$ must be transcendental. Moreover, $\pi$ and $e$ would have to be not only transcendental, but also algebraically independent of each other, which is still an open problem. This is because among $1,i\pi,e^1=e,e^{i\pi}=-1$ there must be at least two algebraically independent ones, but  $1$ and $-1$, being rational, cannot be that. 

The Gelfond-Schneider theorem also easily follows from the Schanuel conjecture. Indeed, consider numbers $\ln a$, $b\ln a$, where $a,b$ are both algebraic, $a\neq0,1$ and $b$ is irrational. Since $a\neq1$ and $b\neq0$ our numbers are not $0$, and since $b$ is irrational they are linearly independent over $\Q$. By the Schanuel conjecture, there must be at least two algebraically independent numbers among $\ln a,b\ln a,e^{\ln a}=a,e^{b\ln a}=a^b$. But $a$ is algebraic by assumption, and $b\ln a$ is an algebraic multiple of $\ln a$, so it is not independent. Thus, the second algebraically independent number can only be $a^b$. In particular, it must be transcendental. More non-trivial consequences of the Schanuel conjecture are derived in \cite{Cheng,Chow,Terzo}.

\section{Algebraically based numbers}\label{ExpLog}

To apply the Schanuel conjecture to our case, it is convenient to extend the tower $\SA$ to a field constructible in terms of exponentials and logarithms. Due to Euler's formula \eqref{exp-1}, we can express $\sin(\pi x)$ algebraically in terms of $(-1)^x$, and $\frac1\pi\arcsin(x)$ in terms of $\log_{\,-1} x$. So all numbers in $\SA$, and then some, can be generated with $(-1)^x$ and $\log_{\,-1} x$. Before, we could not generate $e^\pi$ from the rational numbers because only real $x$ were allowed in $\sin(\pi x)$, but $e^\pi=(-1)^{-i}$, so now we can. There is a technical issue with logarithms of complex numbers---they are multivalued. We can take their principal values, and often we will, but sometimes algebra works out better if we adjoin all their multiple values. This will be indicated by capitalized $\Ln$ or $\Log$.

While we are at it, instead of taking only square roots, we will allow taking radicals of any order, and even roots of polynomial equations with previously generated coefficients that are not solvable in radicals. When applied to a field $F$, the result is called its algebraic closure and denoted $\overline{F}$. We already used this notation for algebraic numbers $\overline{\Q}$. Now the exponential-logarithmic extension of $\SA$ can be defined by analogy to Definition \ref{SAtowers}. This takes us into exponential algebra. We will give a general definition for an arbitrary base $b$ rather than just $b=-1$ because it turns out that we get the same extension for any algebraic $b\neq0,1$.
\begin{definition}\label{ELtowers} Let $b\neq0,1$ be a complex number. Define $\EL_0\!:=\,\Q$, and then recursively $\EL_{n}\!:=\overline{\EL_{n-1}\left(b^{\EL_{n-1}},\Log_{\,b}(\EL_{n-1})\right)}$. We denote $\EL:=\bigcup_{n\geq0}\EL_n$. Equivalently, $\EL$ is the smallest algebraically closed subfield $F\subseteq\C$ such that $x\in F$\, if and only if\, $b^x\in F$. We will write $\EL^{(b)}$ to indicate the base $b$ explicitly.
\end{definition}
\noindent Algebraic numbers are in $\EL^{(b)}$ for any $b$, in fact, $\oQ\subseteq\EL^{(b)}_1$. The tower $\EL^{(e)}$ was introduced by Ritt to formalize the concept of ``elementary numbers" by analogy to the elementary functions of Liouville \cite{Ritt}. A somewhat smaller tower, where one adjoins radicals of all orders at each step rather than takes full algebraic closures, is considered in \cite{Chow}. Its elements are called ``closed-form numbers". Clearly, $e$ is ``elementary," and so is $\pi$ since $\ln(-1)=i\pi$. The next proposition collects some simple properties of $\EL^{(b)}$.
\begin{proposition}\label{ELprop}$\phantom{a}$ Let $x,y,a,b,c\in\C\backslash\{0,1\}$.
\smallskip

\noindent \textup{(i)} If $x,y\in\EL^{(b)}$ then $x^y$ and $\Log_xy\in\EL^{(b)}$;
\smallskip

\noindent \textup{(ii)} If $a\in\EL^{(b)}$ then $\EL^{(a)}\subseteq\EL^{(b)}$;
\smallskip

\noindent \textup{(iii)} $\ds{\EL^{(b)}=\bigcap_{\substack{c\in\C\backslash\{0,1\},\\b\in\EL^{(c)}}}\!\!\!\!\!\!\EL^{(c)}}$;
\smallskip

\noindent \textup{(iv)} Either $\EL^{(c)}\subseteq\EL^{(b)}$ or $\log_cb\not\in\EL^{(b)}$. Moreover, if $\log_cb\not\in\EL^{(b)}$ then at most one of $a$ and $\log_ca$ is in $\EL^{(b)}$ for any $a\neq0,1$. 
\end{proposition}
\begin{proof}$\phantom{a}$ 

\noindent \textup{(i)} By assumption and the definition of $\EL^{(b)}$, we have $\log_b x\in\EL^{(b)}$ and $y\log_b x\in\EL^{(b)}$ since it is a field. Therefore, $x^y=b^{\,y\log_b x}\in\EL^{(b)}$. Moreover, by the change of base formula and closure under division, $\Log_xy=\frac{\Log_by}{\Log_bx}\in\EL^{(b)}$.
\smallskip

\noindent \textup{(ii)} By (i), if $x\in\EL^{(b)}$ then $a^x\in\EL^{(b)}$ and if $a^x\in\EL^{(b)}$ then $x\in\Log_aa^x\subseteq\EL^{(b)}$. Since $\EL^{(a)}$ is the smallest algebraically closed subfield satisfying this biconditional it must be contained in $\EL^{(b)}$.
\smallskip

\noindent \textup{(iii)} Clearly, $\ds{\!\!\!\!\!\!\bigcap_{\substack{c\in\C\backslash\{0,1\},\\b\in\EL^{(c)}}}\!\!\!\!\!\!\EL^{(c)}\subseteq\EL^{(b)}}$ since $b\in\EL^{(b)}$. But by (ii), $\ds{\EL^{(b)}\subseteq\!\!\!\!\!\!\bigcap_{\substack{c\in\C\backslash\{0,1\},\\b\in\EL^{(c)}}}\!\!\!\!\EL^{(c)}}$.

\noindent \textup{(iv)} We will show that $\EL^{(c)}\subseteq\EL^{(b)}$ is equivalent to $\log_cb\in\EL^{(b)}$, which implies the alternative. If $\log_cb\in\EL^{(b)}$ then $c=b^{\frac1{\log_c b}}\in\EL^{(b)}$ and $\EL^{(c)}\subseteq\EL^{(b)}$ by (ii). Conversely, if $\EL^{(c)}\subseteq\EL^{(b)}$ then $b,c\in\EL^{(b)}$ and $\log_cb\in\EL^{(b)}$ by (i). Similarly, if $a,\log_ca\in\EL^{(b)}$ then $c=a^{\frac1{\log_c a}}\in\EL^{(b)}$ by (i) and $\EL^{(c)}\subseteq\EL^{(b)}$ by (ii), which is excluded by the second alternative.
\end{proof}
\noindent Since $\oQ\subseteq\EL^{(b)}$ for any $b\neq0,1$ it follows from Proposition \ref{ELprop}\,(ii) that all $\EL^{(b)}$ for any algebraic $b\neq0,1$ are the same, as we mentioned above. 
\begin{definition}\label{AlgBased} The common field $\EL^{(b)}$ for all $b\in\oQ\backslash\{0,1\}$ will be denoted $\EL^{alg}$ and its elements called algebraically based numbers.
\end{definition}
\noindent Algebraically based numbers are a very natural generalization of algebraic numbers. By Proposition \ref{ELprop}\,(i), they are exactly the numbers that can be obtained from the rational numbers by taking algebraic closures, and exponents and logarithms of previously generated numbers. Moreover, they are contained in $\EL^{(b)}$ for any non-algebraic $b$. In this sense, algebraically based transcendental numbers are the `most algebraic' of transcendental numbers. 

Algebraically based numbers also form a subset of Ritt's ``elementary" numbers \cite{Chow,Ritt}. Since $\SA\subseteq\EL^{(-1)}=\EL^{alg}$, if we can show that $\pi$ is not algebraically based then $\pi\not\in\SA$ and the circle cannot be squared with straightedge, compass and anglesectors. The next theorem raises the stakes beyond this geometric application.
\begin{theorem}\label{PiAlt} Either $\pi$ is not algebraically based or all Ritt's ``elementary" numbers $\EL^{(e)}$ are algebraically based. In the former case, at most one of $a,\ln a$ is algebraically based for all $a\neq0,1$.
\end{theorem}
\begin{proof} Let us show that $\EL^{(e)}=\EL^{(\pi)}$, i.e., one can use $\pi$ instead of $e$ to generate ``elementary" numbers. We already know that $\pi=\frac1i\ln(-1)\in\EL^{(e)}$, but also $e=(-1)^{\frac1{i\pi}}\in\EL^{(\pi)}$, so the conclusion follows by Proposition \ref{ELprop}\,(ii). Therefore, if $\pi\in\EL^{alg}$ then $\EL^{(e)}=\EL^{(\pi)}\subseteq\EL^{alg}$. The last claim follows from the second alternative in Proposition \ref{ELprop}\,(iv) with $c=e$.
\end{proof}
\noindent In other words, if $\pi$ is algebraically based then the distinction between algebraically based and ``elementary" numbers will collapse, while if it is not we will have a great deal of non-algebraically based ``elementary" numbers, namely, the natural logarithms of all algebraically based ``elementary" numbers. Ritt conjectured back in 1948 that the real root $\rho$ of $e^x+x=0$ is not ``elementary" \cite{Ritt}. Since $e^\rho=-\rho$ we have $e=(-\rho)^{\frac1{\rho}}\in\EL^{(\rho)}$, so $\EL^{(e)}\subseteq\EL^{(\rho)}$. Therefore, if Ritt's conjecture is false then $\EL^{(e)}=\EL^{(\rho)}$, and yet another distinction will collapse. 

In 1983, Ferng Lin proved that $\rho\notin\EL^{(e)}$ and $\EL^{(e)}$ is a proper subset of $\EL^{(\rho)}$, but only conditionally on the Schanuel conjecture \cite{Lin}. Even earlier, in 1966, Serge Lang conjectured that $\pi$ is not in the purely exponential tower with base $e$, but again, was only able to prove it conditionally on the Schanuel conjecture \cite{Lang71}. We wish to show, similarly, that $\pi\notin\EL^{alg}$, and will also use the Schanuel conjecture to do it. Thus, under the Schanuel conjecture, there is a hierarchy of transcendental numbers with strict inclusions, $\EL^{alg}\subsetneqq\EL^{(e)}\subsetneqq\EL^{(\rho)}$, while without it they may all be the same field.

\section{No ladder to $\boldsymbol{\pi}$}\label{piLadder}

In this section, we will show first that the Schanuel conjecture implies that $\ln b\not\in\EL^{(b)}=\EL^{alg}$ for any algebraic $b\neq0,1$. Applying it to $b=-1$ entails that $\pi$ is not algebraically based, and hence Conjecture \ref{ConEL}, since $\ln(-1)=i\pi$ and $i$ is algebraic.

Let us fix $b$ for the duration and drop it from our notation. We will first show that any element of $\EL$ can be generated by iteratively adjoining finitely many exponentials and logarithms to $\Q$ (descent). This adjoining can be trimmed to fit into the template of the Schanuel conjecture (reduction). Its application (ascent) will then show that $\ln b$ remains independent of all elements in $\EL$, and hence does not belong to $\EL$. This is enough for our purposes, but an even stronger claim will be a simple consequence.

We will need the following elementary property of algebraic closures. Given a field $F$ and a set of adjoined elements $A$, we have $\overline{\overline{F}(A)}=\overline{F(A)}$. This is because both sides are the smallest algebraically closed fields that contain both $F$ and $A$. This allows us to simplify the definition of $\EL_n$ somewhat. Let $X_i:=b^{\EL_{i}}\cup\Log_{\,b}(\EL_{i})$. Then $\EL_1=\overline{\EL_0\left(X_0\right)}=\overline{\Q\left(X_0\right)}$, and, by induction,
\begin{align}\label{ELnQ}
\EL_{n}&=\overline{\EL_{n-1}(X_{n-1})}=\overline{\overline{\Q(X_{n-2})}\,(X_{n-1})}=\overline{\Q(X_{n-2})\,(X_{n-1})}=\overline{\Q(X_{n-1})}\notag\\
&=\overline{\Q\left(b^{\EL_{n-1}},\Log_{\,b}(\EL_{n-1})\right)}
\end{align}
since $X_{n-2}\subseteq X_{n-1}$. In other words, $\EL_n$ can be generated by adjoining new elements directly to $\Q$ rather than to $\EL_{n-1}$.

To track the adjoining of elements, a construction due to Lin \cite{Lin} is convenient, which we will call a ``ladder." He called it ``graded sequence" and it is analogous to the ``tower" in \cite{Chow}.
\begin{definition}[\textbf{Ladder}]\label{ladder} Let $a_1,\dots,a_m\in\C$. Define $F_0:=\Q$ and for $k\geq1$ 
$$
F_k:=\Q\left(a_1,\dots,a_k,b^{a_1},\dots,b^{a_k}\right).
$$
We call $a_1,\dots,a_m$ a ladder when for every $k\geq1$ either $a_k\in\ov{F_{k-1}}$ or $b^{a_k}\in\ov{F_{k-1}}$. We say that it is a ladder to $\xi\in\EL$, or that $\xi$ is generated by it, when $\xi\in\ov{F_m}$.
\end{definition}
\noindent The idea is that each adjoined number is either itself algebraic over the preceding field or is the logarithm of one. Thereby, we can generate iterated logarithms of the initial numbers. Since the exponentials are explicitly adjoined, we can generate any mixed iterations of them and logarithms. And since algebraic closures are taken at each step we can generate their iterated algebraic combinations as well. Intuitively, we can build a ladder to any element of $\EL$. This intuition is formalized in the next lemma.
\begin{lemma}[\textbf{Descent}]\label{descent} For any $\xi\in\EL$ there is a ladder to $\xi$.
\end{lemma}
\begin{proof} By \eqref{ELnQ}, for any $\xi\in\EL_n$ there is a polynomial with coefficients in $\Q\left(b^{\EL_{n-1}},\Log_{\,b}(\EL_{n-1})\right)$ that it is a root of. Since there are only finitely many coefficients, they are all algebraic combinations of $b^{z_1},\dots,b^{z_r}$ with $z_j\in\EL_{n-1}$ and $w_1,\dots,w_s$ with $b^{w_j}\in\EL_{n-1}$ (i.e., $w_j$ are logarithms with base $b$ of elements in $\EL_{n-1}$). By construction, 
$$
\xi\in\ov{\Q\left(z_1,\dots,z_r,w_1,\dots,w_s,b^{z_1},\dots,b^{z_r},b^{w_1},\dots,b^{w_s}\right)}.
$$
We then apply the same procedure to each $z_j$ and $b^{w_j}$, and continue until we get to elements in $\EL_0=\Q$. Collecting all $z$-s and $w$-s so produced gives us a ladder to $\xi$.
\end{proof}
The ladders produced by descent can have many redundant rungs, as the preceding fields may already contain the elements needed to proceed. To apply the Schanuel conjecture, we need to trim them to weed out at least linear redundancy. In other words, we will show that ladders with linearly independent elements suffice. Moreover, these elements can be chosen linearly independent of $1$ (and hence of $\Q$) as well. This last observation will be instrumental in showing that $\ln b$ is not generated by any ladder.
\begin{lemma}[\textbf{Reduction}]\label{linred} Any ladder to $\xi$ can be reduced to a ladder with $1,a_1,\dots,a_m$ linearly independent over $\Q$.
\end{lemma}
\begin{proof}
Suppose $a_k$ is a linear combination with rational coefficients of the preceding elements of the ladder, i.e., $a_k=q+\sum_{j=1}^{k-1}q_ja_j$ with $q,q_j\in\Q$. Then $a_k\in F_{k-1}$ and 
$$
b^{a_k}=b^q\prod_{j=1}^{k-1}\big(b^{a_j}\big)^{q_j}\in\ov{F_{k-1}},
$$
because $b^q\in\oQ$, and rational powers of field elements are in its algebraic closure. Therefore, we can remove $a_k$ from the ladder and it will still remain a ladder. The corresponding fields will not be affected, except for reindexing, so $\xi$ is still generated by the ladder without $a_k$. Repeating the process, if necessary, we produce a reduced ladder to $\xi$.
\end{proof}
For the next lemma we will finally have to use the Schanuel conjecture. Assuming it, the ascent along a reduced ladder produces fields that are algebraically independent of $\ln b$. By the previous lemma, the same is true for unreduced ladders as well, but we need linear independence over $\Q$ to apply the conjecture at each step. To wit, climbing up any ladder does not bring us to $\ln b$, that is to $\pi$ when $b=-1$.
\begin{lemma}[\textbf{Ascent}]\label{ascent} Assuming the Schanuel conjecture, for a reduced ladder exactly one of each $a_k$, $b^{a_k}$ is transcendental over $F_{k-1}$, and those transcendental numbers are algebraically independent of each other and $\ln b$. In particular, $\ln b\not\in\ov{F_m}$.
\end{lemma}
\begin{proof} Let $a_1,\dots,a_m$ be a reduced ladder so that $1,a_1,\dots,a_m$ are linearly independent over $\Q$. We will proceed by induction on $k$.

Since $1,a_1$ are linearly independent so are $\ln b, a_1\ln b$. Therefore, by the Schanuel conjecture, at least two of $\ln b, a_1\ln b,e^{\ln b}=b,e^{a_1\ln b}=b^{a_1}$ are algebraically independent. But $b$ is algebraic, and either $a_1$ or $b^{a_1}$ is in $\ov{F_0}=\oQ$ by definition of a ladder. In the former case, $a_1\ln b$ is an algebraic multiple of $\ln b$, hence algebraically dependent on it, so $\ln b$, $b^{a_1}$ must be algebraically independent. In the latter case, $\ln b, a_1\ln b$ must be algebraically independent, and hence so are $\ln b, a_1$. In particular, $\ln b\not\in\ov{F_1}$. This establishes the base of induction.

Suppose our claim holds up to $k-1$. Since $1,a_1,\dots,a_k$ are  linearly independent, so are $\ln b,a_1\ln b,\dots,a_k\ln b$, and, by the Schanuel conjecture, at least $k+1$ of 
$$
\ln b,a_1\ln b,\dots,a_k\ln b,\,b,\,b^{a_1},\dots,\,b^{a_k}
$$
are algebraically independent. We can drop the $\ln b$ multiples from $a_j$ because $\ln b$ is the first number on the list and this does not affect algebraic dependence. By the induction hypothesis, for $j\leq k-1$ we can define $b_j:=a_j$ when $a_j$ is transcendental over $F_{j-1}$, and $b_j:=b^{a_j}$ when $b^{a_j}$ is. Since the other numbers in each pair are algebraic in the preceding $b_j$, the algebraically independent $k+1$ numbers must be among $\ln b,b_1,\,\dots,\,b_{k-1},\,a_k,\,b^{a_k}$. 

By definition of a ladder, at least one of $a_k,\,b^{a_k}$ is algebraic over $F_{k-1}$, so, to get to $k+1$, exactly one of them must be transcendental over $F_{k-1}$. We can denote it $b_k$, in agreement with the previous notation. Then $\ln b,b_1,\,\dots,\,b_k$ are algebraically independent, $\ov{F_k}=\ov{F_{k-1}(b_k)}$, and $\ln b\not\in\ov{F_k}$ since it is transcendental over $F_{k-1}$ and algebraically independent of $b_k$. This completes the induction.
\end{proof}
Our main theorem is now a simple consequence, and so are Conjecture \ref{ConEL} and the negative answer to Question \ref{QRRA}. 
\begin{theorem}\label{pinEL} Assuming the Schanuel conjecture, $\pi$ and $e$ are not algebraically based, and neither is $\ln a$ for any algebraically based $a\neq0,1$.
\end{theorem}
\begin{proof}
It is a direct consequence of Lemmas
\ref{descent}--\ref{ascent} that $\ln b\not\in\EL^{(b)}$. Indeed, any element of $\EL^{(b)}$ has a ladder to it, even a reduced one, but $\ln b$ does not. Therefore, $\ln(-1)=i\pi\not\in\EL^{(-1)}=\EL^{alg}$. The remaining claims follow from Theorem \ref{PiAlt}. In particular $e\notin\EL^{alg}$ because $\ln e=1\in\EL^{alg}$.
\end{proof}
\noindent The conclusion of the theorem is a strengthening of Lindemann's and Hermite's results that $\pi$ and $e$ are not algebraic. We cannot generate them starting from rational numbers even if we exponentiate and take logarithms in addition to taking algebraic closures. Natural logarithms of integers and rational numbers, like $\ln 2$, are also beyond reach by these means. All of them are examples of non-algebraically based ``elementary" numbers. Alas, all of this is only conditional.

\section{Related work and open questions.}\label{Conc}

Even if the Schanuel conjecture is true we still cannot answer the title question with a definitive no. This is because our approved uses of the quadratrix were too restrictive. Geometrically, we only allowed intersecting it with horizontal and radial lines. Indeed, the Cartesian equation of the quadratrix (for the radius $1$) is $y=x\tan\big(\frac{\pi}{2}y\big)$. Intersecting it with a horizontal line $y=b$ geometrically produces $x=b/\tan\big(\frac{\pi}{2}b\big)$. And intersecting it with the radial line $y=mx$ produces $y=\frac2\pi\arctan(m)$. These are algebraically related to $\sin(\pi b)$ and $\frac1\pi\arcsin(m)$, the functional counterparts of the right anglesector and its reverse from Theorem \ref{AnalAngsec}. 

Intersecting the quadratrix with more general lines already leads to new numbers. For example, intersecting it with a vertical line $x=a$ would produce roots of a transcendental equation $y-a\tan\big(\frac{\pi}{2}y\big)=0$. This is a kind of equation whose solutions are not even ``elementary" in Ritt's sense, assuming the Schanuel conjecture \cite{Chow,Lin}. In exponential algebra, one can define an extension of $\EL^{alg}$ that would cover intersections not only with all lines, but also with circles, conic sections, other quadratrices, and even with some other transcendental curves, like Archimedean spirals. This is the field of {\it exponentially algebraic numbers}, which are roots of any non-degenerate exponential-logarithmic systems of equations, not just algebraic ones \cite{Kir10}. We will go out on a limb of geometric intuition and conjecture that even those do not suffice to square the circle.
\begin{conjecture}\label{ConELhat} $\pi$ is not exponentially algebraic.
\end{conjecture}
\noindent Proving it seems to be much harder than our Theorem \ref{pinEL}, even assuming the Schanuel conjecture. The best result in this direction is due to Giuseppina Terzo \cite{Terzo}, who proved conditionally that the only non-trivial relation between $\pi$ and $i$ in exponential {\it rings} is Euler's $e^{i\pi}=-1$. If $\pi$ was a root of some exponential equation, like $x+(-1)^x=0$, there would be other relations. Unfortunately, exponential rings only allow integer coefficients in equations, which is a severe restriction.

There are analogs of straightedge and compass constructions and the quadrature problem for them in hyperbolic geometry \cite{Jagy}. The 1948 impossibility result for it is due to Nikolai Nestorovich. We are not aware of hyperbolic analogs of the quadratrix, but one can define hyperbolic analogs of the right anglesector and its reverse and ask the question about squaring hyperbolic circles with them. One would expect hyperbolic functions to play a role analogous to trigonometric functions in the Euclidean quadrature, but the impossibility of quadrature is unclear because the corresponding tower might be $\EL^{(e)}$ instead of $\EL^{alg}$.

The structure of the field of algebraically based numbers $\EL^{alg}$ is of interest in its own right. We proved (conditionally) that at most one of $a$ and $\ln a$ is algebraically based, but it may well be neither for cardinality reasons. In its iterative construction, only countably many elements are added to $\Q$ at each step, so $\EL^{alg}$ is countable. But $\R$ and $\C$ are not, so for `almost all' real or complex $a$, neither $a$ nor $\ln a$ is algebraically based. More explicitly, complex solutions to $x=\ln x$ are not  algebraically based. But what if we restrict the question to Ritt's ``elementary" numbers? In particular, are $\ln\pi$ or $e^e$  algebraically based, assuming the Schanuel conjecture? 

The exponential-logarithmic towers $\EL^{(b)}$ present other interesting algebraic questions. They are partially ordered by inclusion, which arranges numbers $b$ into an exponential-logarithmic hierarchy of transcendence. What kind of hierarchy is it? For example, are any two $\EL^{(a)}$ and $\EL^{(b)}$ always contained in some $\EL^{(c)}$? Recall that under the Schanuel conjecture $\EL^{alg}\subsetneqq\EL^{(e)}$. Are their transcendental numbers $b$ intermediate between algebraically based and ``elementary" numbers, i.e., such that $\EL^{alg}\subsetneqq\EL^{(b)}\subsetneqq\EL^{(e)}$?

In recent years, there has been some progress towards proving the Schanuel conjecture itself. It builds on the work of Angus Macintyre and Alex Wilkie in 1990's that applied methods of the model theory from mathematical logic to transcendental number theory. More recently, Boris Zilber axiomatized properties of $\C$ with the standard exponentiation, and proved that there exists a unique exponential field satisfying them and the Schanuel conjecture. A natural conjecture is that it is $\C$ itself \cite{Zil05}. Some of the model-theoretic results were reproduced algebraically and extended by Kirby, who was able to prove that the Schanuel conjecture has at most countably many ``essential" counterexamples (those not induced by simpler ones) \cite{Kir10}.

It is remarkable that some fairly simple geometric constructions still lead to open questions in transcendental number theory. After 2,500+ years, the quadrature remains defiant! We hope that the readers will be encouraged to take up its challenges at the intersection of geometry and algebra.

\bibliographystyle{vancouver}

\end{document}